\documentclass[11pt]{article}
\usepackage[T1]{fontenc}
\usepackage{lmodern}
\usepackage[letterpaper,margin=1.05in]{geometry}
\usepackage{microtype}
\usepackage{amsmath,amssymb,amsthm,mathtools}
\usepackage{dsfont}
\usepackage{enumitem}
\usepackage{aliascnt}
\usepackage{xcolor}
\usepackage{tikz}
\usepackage[normalem]{ulem}
\usepackage[colorlinks=true,linkcolor=blue!45!black,citecolor=green!30!black,
  urlcolor=blue!50!black,breaklinks=true]{hyperref}
\usepackage[nameinlink,capitalize,noabbrev]{cleveref}
\numberwithin{equation}{section}
\newtheorem{theorem}{Theorem}[section]
\newaliascnt{proposition}{theorem}
\newtheorem{proposition}[proposition]{Proposition}
\aliascntresetthe{proposition}
\newaliascnt{lemma}{theorem}
\newtheorem{lemma}[lemma]{Lemma}
\aliascntresetthe{lemma}
\newaliascnt{corollary}{theorem}
\newtheorem{corollary}[corollary]{Corollary}
\aliascntresetthe{corollary}
\theoremstyle{definition}
\newaliascnt{problem}{theorem}
\newtheorem{problem}[problem]{Problem}
\aliascntresetthe{problem}
\theoremstyle{remark}
\newaliascnt{remark}{theorem}
\newtheorem{remark}[remark]{Remark}
\aliascntresetthe{remark}
\crefname{problem}{Problem}{Problems}
\newcommand{\R}{\mathbb R}
\newcommand{\E}{\mathbb E}
\newcommand{\PP}{\mathbb P}
\newcommand{\one}{\mathbf 1}
\newcommand{\cE}{\mathcal E}
\newcommand{\cN}{\mathcal N}
\newcommand{\ip}[2]{\langle #1,#2\rangle}
\newcommand{\norm}[1]{\lVert #1\rVert}
\newcommand{\ind}{\mathds 1}
\DeclareMathOperator{\diag}{diag}
\DeclareMathOperator{\rank}{rank}
\DeclareMathOperator{\tr}{tr}
\DeclareMathOperator{\conv}{conv}
\DeclareMathOperator{\Cov}{Cov}
\DeclareMathOperator{\Var}{Var}
\title{Stochastic Domination of Gaussian Maxima\\by the Regular Simplex}
\author{Abhijeet Mulgund\thanks{Support from grants NSF 2217023 and NSF 2240532
is acknowledged.}\\
  \small University of Illinois Chicago\\
  \small \href{mailto:mulgund2@uic.edu}{\texttt{mulgund2@uic.edu}}}
\date{}
\hypersetup{pdftitle={Stochastic Domination of Gaussian Maxima by the Regular Simplex},
  pdfauthor={Abhijeet Mulgund}}

\begin{document}
\maketitle

\begin{abstract}
Let \(n\ge2\), and let \(X=(X_1,\ldots,X_n)\) be a centered Gaussian vector with
\(\Var(X_i)=1\) for every \(i\). Let \(Z_1,\ldots,Z_n\) be independent standard Gaussians, and put \(\overline Z=(Z_1+\cdots+Z_n)/n\). We prove
\[
 \PP\Bigl\{\max_i X_i\le t\Bigr\}\ \ge\
 \PP\Bigl\{\sqrt{\tfrac n{n-1}}\,\max_i\,(Z_i-\overline Z)\le t\Bigr\}
 \qquad\text{for every }t\in\R,
\]
and for each fixed \(t>0\) equality holds only when
\(\Cov(X_i,X_j)=-1/(n-1)\) for all \(i\ne j\). The right side is the
distribution function of the maximum of the regular simplex vector.
Equivalently, among all simplices containing a given centered ball, the
regular simplex circumscribed about the ball has the least standard Gaussian
measure, as conjectured by Balitskiy, Karasev, and Tsigler. In our preceding
paper we proved this comparison after both maxima are smoothed by independent
Gaussian noise of variance \(1/(n-1)\), which suffices for the Weak Simplex
Conjecture; here we remove the smoothing, which is what probabilities at a
single threshold require. As an application we consider \(n\) equally likely signals of equal energy in Gaussian noise, where the transmitter may also send nothing. At every positive false-alarm level, and for every law of a common nonnegative random amplitude not concentrated at zero, the regular simplex uniquely maximizes the average probability of correct identification whenever the signal dimension is at least \(n-1\).
A Lean formalization is available at \url{https://github.com/abhmul/full-simplex-conjecture-lean}.
\end{abstract}



\smallskip
{\small\noindent\textbf{AI disclosure.} Generative AI provided the key
insights of the argument, with some input from the author. AI was also used
for proof development and checking, literature discovery, drafting,
editorial revision, and Lean formalization. The author independently checked
the mathematical arguments and references and takes full responsibility for
the correctness, attribution, and presentation of the paper. No theorem,
proof step, or bibliographic assertion was retained solely on the authority
of an automated system. No automated system is an author, and the proofs do
not rely on a computer-assisted verification. Further details can be found in
\cref{app:ai}.\par}

\section{Introduction}\label{sec:introduction}

In 2017 Balitskiy, Karasev, and Tsigler asked which simplex containing a given
ball has the least Gaussian measure \cite[Conjecture~3.3]{BKT2017}. Their
conjecture is that among all simplices in \(\R^{n-1}\) that contain the
centered ball \(B(0,t)\), the regular simplex circumscribed about the ball is
the minimizer. For tetrahedra (\(n=4\)) this already follows from a spherical
theorem of Fejes T\'oth \cite[p.~27, (4)]{FejesToth1950}; see also
\cite[p.~4]{BKT2017}. We prove the conjecture in every dimension, together with
its equality case.

The conjecture can be interpreted as a statement about maxima of Gaussian
random variables. A simplex containing \(B(0,t)\) can be written as
\(\{y:\ip{v_i}{y}\le b_i,\ 1\le i\le n\}\) with unit outer normals
\(v_1,\ldots,v_n\) and offsets \(b_i\ge t\). If \(W\) is a standard Gaussian
vector in \(\R^{n-1}\), its Gaussian measure is at least
\[
 \PP\{\ip{v_i}{W}\le t\text{ for every }i\}
 =\PP\Bigl\{\max_i X_i\le t\Bigr\},\qquad X_i=\ip{v_i}{W}.
\]
The \(X_i\) are standard Gaussians whose correlation matrix is the Gram matrix
of the normals. For the regular simplex, \(X\) has the law of the regular
simplex vector
\begin{equation}\label{eq:simplex-vector}
 \xi_i=\sqrt{\frac n{n-1}}\,\bigl(Z_i-\overline Z\bigr),
 \qquad \overline Z=\frac1n\sum_{j=1}^nZ_j,
\end{equation}
where \(Z_1,\ldots,Z_n\) are independent standard Gaussians. Its coordinates
have unit variance, a common correlation \(-1/(n-1)\), and sum to zero. The
conjecture therefore follows if
\[
 \PP\Bigl\{\max_iX_i\le t\Bigr\}\ \ge\ \PP\Bigl\{\max_i\xi_i\le t\Bigr\}
 \qquad(t\in\R)
\]
for every centered Gaussian vector \(X\) with unit variances, whatever its
correlations. \Cref{thm:main} establishes precisely this inequality.

The same maximum governs a classical problem in communication. If \(n\)
equally likely signals \(\lambda v_1,\ldots,\lambda v_n\) of equal energy are
sent through additive white Gaussian noise, maximum-likelihood decoding
succeeds with probability
\(n^{-1}e^{-\lambda^2/2}\,\E\exp(\lambda\max_i\ip{v_i}{W})\), where \(W\) is a
standard Gaussian vector \cite[p.~488, (2.3)]{Balakrishnan1961}. The Weak
Simplex Conjecture, which goes back to a remark of Shannon reported by Rice
\cite[p.~68]{Rice1950} and was named by Massey \cite{Massey1988}, asserts that
the regular simplex maximizes this probability. Balitskiy, Karasev, and
Tsigler proposed their conjecture as a route to it: slicing the decoding
integral at each level shows that their conjecture implies the Weak Simplex
Conjecture \cite[Lemma~3.1 and pp.~3--4]{BKT2017}. In our preceding paper
\cite{Mulgund2026} we instead resolved the Weak Simplex Conjecture by proving
a comparison of Gaussian maxima after smoothing, which we describe in
\cref{sec:smoothed}. \Cref{thm:main} completes the more direct route proposed
by Balitskiy, Karasev, and Tsigler.

\subsection{The main result}

Write \(\one=(1,\ldots,1)^\top\) for the all-ones vector of any dimension,
\(J=\one\one^\top\), \(I\) for the identity matrix (with a subscript when its
dimension matters), and \(g_{ij}\) for the entries of a matrix \(G\). The set
of \(n\times n\) correlation matrices is
\[
 \cE_n=\{G\in\R^{n\times n}:G=G^\top\succeq0,\ g_{ii}=1\}.
\]
For \(G\in\cE_n\), let \(X^G\sim\cN(0,G)\) and define
\begin{equation}\label{eq:definitions}
 f_G(x)=\PP\{X^G_i\le x_i\text{ for every }i\},\qquad
 F_G(t)=f_G(t\one),\qquad M_G=\max_i X^G_i.
\end{equation}
The regular simplex vector \eqref{eq:simplex-vector} has covariance
\begin{equation}\label{eq:simplex}
 \Delta_n=\frac{nI-J}{n-1},
\end{equation}
and we write \(D_n(t)=F_{\Delta_n}(t)\) for the distribution function of its maximum.
The matrix \(\Delta_n\) has rank \(n-1\); it is the Gram matrix of the
vertices of a regular simplex inscribed in the unit sphere of \(\R^{n-1}\).

\begin{theorem}\label{thm:main}
For every integer \(n\ge2\), every \(G\in\cE_n\), and every \(t\in\R\),
\begin{equation}\label{eq:main}
 F_G(t)\ge D_n(t).
\end{equation}
For each fixed \(t>0\),
\begin{equation}\label{eq:main-equality}
 F_G(t)=D_n(t)\quad\Longleftrightarrow\quad G=\Delta_n.
\end{equation}
\end{theorem}

In other words \(M_G\le_{\mathrm{st}}M_{\Delta_n}\), where
\(U\le_{\mathrm{st}}V\) means \(\PP\{U\le t\}\ge\PP\{V\le t\}\) for every
\(t\). The equality statement is restricted to positive thresholds because for
\(n\ge3\) it fails at \(t\le0\). Since \(\sum_i\xi_i=0\), we have \(D_n(t)=0\)
for \(t\le0\), and every \(G\) containing an antipodal pair (\(g_{ij}=-1\) for
some \(i\ne j\)) also has \(F_G(t)=0\) for \(t\le0\).

\Cref{thm:main} resolves the circumscribed-simplex conjecture in the following
form. Let \(\gamma_r\) denote standard Gaussian measure on \(\R^r\).

\begin{corollary}\label{cor:geometric}
Let \(t>0\), and let \(T_t\subset\R^{n-1}\) be a regular simplex circumscribed
about \(B(0,t)\). If a simplex \(K\subset\R^{n-1}\) contains \(B(0,t)\), then
\[
 \gamma_{n-1}(K)\ge\gamma_{n-1}(T_t),
\]
with equality if and only if \(K\) is the image of \(T_t\) under an orthogonal
transformation.
\end{corollary}

The inequality is immediate from \cref{thm:main}. Write
\(K=\{y:\ip{v_i}{y}\le b_i\}\) with unit normals \(v_i\) and offsets
\(b_i\ge t\), and let \(G\) be the Gram matrix of the normals. The facet
normals of \(T_t\) have Gram matrix \(\Delta_n\), so
\begin{equation}\label{eq:geometric-chain}
 \gamma_{n-1}(K)\ \ge\ \gamma_{n-1}\{y:\ip{v_i}{y}\le t\text{ for every }i\}
 =F_G(t)\ \ge\ D_n(t)=\gamma_{n-1}(T_t).
\end{equation}
If equality holds, \cref{thm:main} gives \(G=\Delta_n\), and every \(b_i\)
must equal \(t\); we give the details in \cref{app:equivalence}. The converse
also holds: for each fixed \(n\), the non-strict assertion of
\cref{cor:geometric} for all \(t>0\) implies \eqref{eq:main}. We prove this in
\cref{app:equivalence}, starting from an argument sketched in
\cite[p.~4]{BKT2017}, so \eqref{eq:main} is equivalent to the conjecture of
Balitskiy, Karasev, and Tsigler. In particular, the case \(n=4\) of
\eqref{eq:main} already follows from the theorem of Fejes T\'oth mentioned
above.\footnote{Balitskiy, Karasev, and Tsigler use the Gaussian density
proportional to \(e^{-|x|^2}\); rescaling the radius gives our
normalization.}

\subsection{Relation to the smoothed comparison}\label{sec:smoothed}

The comparison proved in \cite[Theorem~2.1]{Mulgund2026} is
\(F_R(c)\ge\Phi(c)^n\) for every \(c\in\R\) and every \(R\in\cE_n\) with
\(R-J/n\succeq0\), where \(\Phi\) is the standard normal distribution
function. These are precisely the matrices
\[
 R=\frac{(n-1)G+J}{n},\qquad G\in\cE_n,
\]
and this map sends \(\Delta_n\) to \(I_n\). If \(B\sim\cN(0,1)\) is independent
of \(X^G\), then \(\sqrt{(n-1)/n}\,X^G+B\one/\sqrt n\) has covariance \(R\), so
\(M_R\) has the law of \(\sqrt{(n-1)/n}\,M_G+B/\sqrt n\). Dividing by
\(\sqrt{(n-1)/n}\), the theorem of \cite{Mulgund2026} is therefore equivalent to
\begin{equation}\label{eq:predecessor}
 M_G+\sigma B\ \le_{\mathrm{st}}\ M_{\Delta_n}+\sigma B',
 \qquad \sigma=\frac1{\sqrt{n-1}},\qquad G\in\cE_n,
\end{equation}
where \(B,B'\) are standard Gaussians independent of the respective maxima.
That is, the two maxima are compared after both are smoothed by an independent
Gaussian of variance \(1/(n-1)\).

The smoothing is harmless for exponential moments: \(\E e^{\lambda(M+\sigma
B)}=e^{\lambda^2\sigma^2/2}\,\E e^{\lambda M}\), so \eqref{eq:predecessor}
compares \(\E e^{\lambda M_G}\) with \(\E e^{\lambda M_{\Delta_n}}\), which is
what decoding needs. It is \emph{not} harmless for a threshold. Integrating
\eqref{eq:predecessor} against a nondecreasing function \(\chi\) gives
\[
 \E\psi(M_G)\le\E\psi(M_{\Delta_n}),\qquad \psi(x)=\E\chi(x+\sigma B).
\]
Every such \(\psi\) that is finite is also continuous and, unless constant,
strictly increasing. So \(\psi\) is never a threshold indicator
\(\ind_{\{x>t\}}\), and \eqref{eq:predecessor} does not by itself yield
\eqref{eq:main}: smoothing preserves stochastic order, but an order between
smoothed variables need not hold for the unsmoothed ones. \Cref{thm:main}
removes the smoothing. It implies \eqref{eq:predecessor} by conditioning on
\(B\), and it determines equality at a single positive threshold. For the
smoothed comparison, strictness at a single threshold was obtained
independently in our preceding paper
\cite[Theorem~2.1 and Remark~2.3]{Mulgund2026} and by Su et al.\
\cite[Theorem~1]{SuYangXuI2026}.

The only other stochastic comparison with an extremal regular simplex in every
dimension that we know of is also smoothed. Let \(K\subset\R^r\) be a convex
body whose minimal-volume containing ellipsoid is the unit ball, and let
\(p_K\) be its gauge. When \(K=\{x:\ip{v_i}{x}\le1\}\) is a polytope,
\(p_K(W)=\max_i\ip{v_i}{W}\) is a Gaussian maximum. Barthe proved that the
regular simplex \(T\) inscribed in the unit ball maximizes \(\E p_K(W)\)
\cite[Theorem~3]{Barthe1998}, and his proof (pp.~691--692) gives the
stronger comparison
\[
 p_K(W)+\sqrt r\,W_0\ \le_{\mathrm{st}}\ p_T(W)+\sqrt r\,W_0,
\]
where \(W\) is a standard Gaussian vector in \(\R^r\) and \(W_0\) is an
independent standard Gaussian. B\"or\"oczky, Fodor, and Hug display this
inequality in integral form for the convex hull \(C\subseteq K\) of the
finitely many contact points in a John decomposition \cite[(70)]{BFH2021},
which gives it for \(K\) because
\(p_K\le p_C\). Like \eqref{eq:predecessor}, it compares smoothed variables,
here with variance \(r\), and it compares only convex bodies whose minimal-volume containing ellipsoid is the unit ball.

\subsection{The difficulty and the idea}

The difficulty lies in the failure of the usual covariance comparison.
Slepian's inequality \cite{Slepian1962,Plackett1954} gives \eqref{eq:main} whenever
\(g_{ij}\ge-1/(n-1)\) for all \(i\ne j\), since \(\Delta_n\) is then entrywise
smaller than \(G\) off the diagonal. For every other \(G\), however, some pair of
coordinates is more negatively correlated than in the simplex. Since \(G\) is
positive semidefinite,
\begin{equation}\label{eq:average-correlation}
 0\le\one^\top G\one=n+\sum_{i\ne j}g_{ij},
\end{equation}
so the average off-diagonal entry is at least \(-1/(n-1)\), so some other
pair must be less negatively correlated than in the simplex. Thus such a \(G\) is not
comparable with \(\Delta_n\) entrywise, and the theorem must handle precisely
these matrices, in which some very negative correlations are compensated by less negative ones.

We argue by induction on \(n\), through the derivative of \(F_G\) in \(t\).
Let \(n\ge3\) and \(t>0\), and suppose that no two coordinates coincide
(\(g_{ij}<1\) for \(i\ne j\)). Then (\cref{prop:boundary-calculus})
\begin{equation}\label{eq:intro-derivative}
 F_G'(t)=\varphi(t)\sum_{i=1}^nH_i(t),\qquad
 H_i(t)=\PP\{X^G_j\le t\ (j\ne i)\mid X^G_i=t\},
\end{equation}
where \(\varphi=\Phi'\), while for the simplex the classical recurrence
(\cref{lem:recurrence}) gives
\begin{equation}\label{eq:intro-recurrence}
 D_n'(t)=n\varphi(t)D_{n-1}(\rho_n(t)),\qquad \rho_n(t)=t\sqrt{\frac n{n-2}}.
\end{equation}
Comparing \eqref{eq:intro-derivative} with \eqref{eq:intro-recurrence}, it
would suffice to show \(H_i(t)\ge D_{n-1}(\rho_n(t))\) for every \(i\). The
induction hypothesis gives this only when every \(g_{ij}\le-1/(n-1)\): given
\(X^G_i=t\), each other coordinate, standardized, must stay below the offset
\(\beta_{ij}=t\sqrt{(1-g_{ij})/(1+g_{ij})}\), which is at least \(\rho_n(t)\)
exactly when \(g_{ij}\le-1/(n-1)\). By \eqref{eq:average-correlation}, this fails for every
\(G\ne\Delta_n\).

Our main observation is that at a covariance minimizing \(F_G(t)\) the
compensation between correlations suffices \emph{in aggregate}: the
individual bounds may fail, but their sum holds. Since \(\cE_n\) is compact and
\(F_G(t)\) is continuous in \(G\) (\cref{lem:compact}), some \(G\in\cE_n\)
minimizes \(F_G(t)\), and for this minimizing \(G\) we prove
\begin{equation}\label{eq:intro-aggregate}
 \sum_{i=1}^nH_i(t)\ \ge\ nD_{n-1}(\rho_n(t)),
\end{equation}
with equality only at \(\Delta_n\) (\cref{prop:interface}, with the induction
hypothesis). By \eqref{eq:intro-derivative} and \eqref{eq:intro-recurrence},
\eqref{eq:intro-aggregate} is the derivative inequality
\(F_G'(t)\ge D_n'(t)\) at the minimizer, and we show in \cref{sec:global} that
this inequality at minimizers implies \cref{thm:main}, together with its
equality case.

Three ingredients prove \eqref{eq:intro-aggregate}. All of them involve the
pair weights
\[
 q_{ij}=\varphi_2(t,t;g_{ij})\,
 \PP\{X^G_\ell\le t\ (\ell\ne i,j)\mid X^G_i=X^G_j=t\}\qquad(i\ne j),
\]
where \(\varphi_2(\cdot,\cdot;c)\) is the standard bivariate Gaussian density
with correlation \(c\).

\emph{First-order conditions} (\cref{sec:variation}). Add independent Gaussian
noise with covariance \(sL\), where \(L\succeq0\), to \(X^G\) and renormalize
the coordinates to unit variance. This moves \(G\) along a path \(G_s\) in
\(\cE_n\), defined only for \(s\ge0\), and \cref{prop:noise} gives
\[
 \frac{d}{ds}F_{G_s}(t)\Big|_{s=0+}=\frac12\tr(SL),\qquad
 S_{ij}=q_{ij}\ (i\ne j),\quad S_{ii}=-\sum_{j\ne i}g_{ij}q_{ij}.
\]
At a minimizer this one-sided derivative is nonnegative for every
\(L\succeq0\), which means \(S\succeq0\). Taking \(L=G\) reduces \(G_s\) to the constant path \(G_s \equiv G\), so
\(\tr(SG)=0\), and for positive semidefinite \(S\) and \(G\) this forces
\(SG=0\). Every row of \(S\) has a positive off-diagonal entry
(\cref{lem:row-positive}), so \(k=S\one\) is a kernel vector of \(G\) with
positive entries, and \(G\) is singular (\cref{prop:minimizer}). Testing
\(S\succeq0\) on the vectors \(\varepsilon_i-a_i\one\), where
\(a_i=k_i/\sum_jk_j\), gives for each row \(i\) of \(G\) (\cref{lem:row-bound})
\[
 (n-2)\Bigl(1-\frac{A_i}{A_*}\Bigr)\ \ge\ na_i-1,\qquad
 A_i=\frac1{k_i}\sum_{j\ne i}\sqrt{1-g_{ij}^2}\,q_{ij},\quad
 A_*=\sqrt{\frac{n-2}n}.
\]
The right sides sum to zero, so unless every \(a_i=1/n\), some row has a
negative right side: the row inequalities are useful only in aggregate.

\emph{The conditional masses} (\cref{sec:boundary,sec:mass}). Let
\(\beta_i=(\beta_{ij})_j\) be the vector of the offsets above, over \(j\ne i\)
with \(g_{ij}>-1\) (a coordinate with \(g_{ij}=-1\) equals \(-t\) and imposes
no constraint). Let \(h_i(b)\) be the conditional probability in
\eqref{eq:intro-derivative} with a free vector of offsets \(b=(b_j)_j\) in
place of \(\beta_i\), so that \(H_i(t)=h_i(\beta_i)\); as with \(F_G(t)=f_G(t\one)\),
the capital letter is the probability at the common threshold \(t\), while
\(h_i\) takes standardized offsets (so \(H_i(t)\ne h_i(t\one)\) in general). By
\cref{prop:boundary-calculus},
\(\partial_jh_i(\beta_i)=\sqrt{1-g_{ij}^2}\,q_{ij}/\varphi(t)\), so the same
weights govern the conditional masses. Gaussian measure is log-concave in the
offsets, so \(1/h_i\) is convex, and its tangent inequality at \(\beta_i\),
\[
 \frac{H_i(t)^2}{h_i(\rho_n(t)\one)}-H_i(t)\ \ge\
 (\beta_i-\rho_n(t)\one)\cdot\nabla h_i(\beta_i)
\]
(see \eqref{eq:reciprocal}), compares \(H_i(t)\) with the mass
\(h_i(\rho_n(t)\one)\) at equal offsets. The induction hypothesis gives
\(h_i(\rho_n(t)\one)\ge D_{n-1}(\rho_n(t))\). A dilation identity
\eqref{eq:dilation} splits \(H_i(t)\) into two positive parts. Its constants
match those of the row inequalities only in dimension \(n-2\), and the
conditional regions lie in dimension \(\rank G-1\le n-2\) because \(G\) is
singular.

\emph{Aggregation} (\cref{lem:scalar}). A scalar lemma sums the \(n\)
resulting inequalities to \eqref{eq:intro-aggregate} and identifies its
equality case.

\subsection{Consequences and applications}

Integrated quantities such as \(\E M_G\) and \(\E e^{\lambda M_G}\), the latter of which
governs maximum-likelihood decoding, already follow from
\eqref{eq:predecessor}; we recover them in \cref{sec:integrated} as special
cases of \cref{cor:integrated}. However, a hard threshold appears as soon as a
receiver may declare that no signal is present, and \eqref{eq:predecessor}
no longer applies.

In \cref{sec:identification} we allow such an inactive state, bound the
probability of a false alarm (announcing a signal when none was sent, i.e., a
false positive) by \(\alpha\), and ask which signal geometry
maximizes the average probability of correct identification. For unit signal
directions \(u_1,\ldots,u_n\) with Gram matrix \(G\), the optimal receiver
announces a label \(i\) maximizing \(\ip{u_i}{Y}\) when
\(\max_i\ip{u_i}{Y}>t_G(\alpha)\), where \(Y\) is the observation and
\(t_G(\alpha)\) is the upper \(\alpha\)-quantile of \(M_G\). At a
deterministic amplitude \(\lambda>0\) its success probability is the truncated
exponential moment
\[
 C_G(\lambda,\alpha)=\frac{e^{-\lambda^2/2}}n\,
 \E\bigl[e^{\lambda M_G}\ind_{\{M_G>t_G(\alpha)\}}\bigr].
\]
This depends on the tail of \(M_G\) at individual thresholds, which the
smoothed comparison does not determine. Using \cref{thm:main},
\cref{thm:identification} shows that the regular simplex is the unique
optimal geometry at every positive false-alarm level, for every law of a
common nonnegative amplitude not concentrated at zero, whenever the signal
dimension is at least \(n-1\). At \(\alpha=1\) the constraint disappears and
the result reduces to the Weak Simplex Conjecture.

In \cref{sec:graphs} we prove a lower bound on the independence ratio of
arc-transitive graphs of every degree. Harangi and Vir\'ag obtained
independent sets as the strict local maxima of a Gaussian eigenvector
\cite{HV2015}, and showed that for arc-transitive graphs the resulting bound
would follow from a conjectured extremal property of spherical caps
\cite[Section~2.2]{HV2015}, which we recall in \cref{sec:open}; they proved
the bound for cubic vertex-transitive graphs and for degree-four arc-transitive graphs. \Cref{cor:graph} proves the
bound in every degree from \eqref{eq:predecessor}. \Cref{thm:main} gives the underlying bound on the
probability of a local maximum at every eigenvalue (\cref{app:graphs}),
although only the least eigenvalue matters for independent sets.

\subsection{Organization}

The rest of the paper follows the proof. \Cref{sec:global} reduces
\cref{thm:main} to the aggregate bound \eqref{eq:intro-aggregate} at a
minimizing covariance (\cref{prop:interface}), using compactness and a
joint-minimum argument. \Cref{sec:boundary} computes the first and second
threshold derivatives of \(f_G\) in terms of the conditional masses and the
pair weights. \Cref{sec:variation} derives the first-order conditions
\(S\succeq0\) and \(SG=0\) at a minimizer and turns them into one inequality
for each row of \(G\). \Cref{sec:mass} combines log-concavity, dilation, and
the scalar lemma into the aggregate bound and its equality case.
\Cref{sec:consequences} contains the applications and \cref{sec:open} two open
problems. \Cref{app:equivalence} completes the proof of
\cref{cor:geometric} and proves its converse, \cref{app:identification}
collects refinements of the identification theorem, \cref{app:graphs}
gives the local comparison at every eigenvalue, and \cref{app:ai} describes
how AI was used.

The proof given here is complete, but it is presented as we first came to
understand it. We expect to find better ways to frame several of its steps and
will revise the exposition in later versions. The author welcomes discussion of
the proof and would be glad to collaborate with anyone interested in
simplifying it or understanding it more deeply.

\section{Reduction to a derivative bound at a minimizer}\label{sec:global}

In this section we reduce \cref{thm:main} to \cref{prop:interface}. Let
\(n\ge3\) and \(t>0\), and assume \eqref{eq:main} in dimension \(n-1\). Then
\cref{prop:interface,lem:recurrence} give \(F'_G(t)\ge D'_n(t)\) for every
covariance \(G\) minimizing \(\Gamma\mapsto F_\Gamma(t)\) on \(\cE_n\).
\Cref{prop:interface} is proved in
\cref{sec:boundary,sec:variation,sec:mass}.

To minimize over covariances we need compactness and continuity.

\begin{lemma}\label{lem:compact}
The set \(\cE_n\) is compact, and \((G,x)\mapsto f_G(x)\) is continuous on
\(\cE_n\times\R^n\). For every \(t>0\),
\[
 F_G(t)\ge\gamma_n(B(0,t))>0\qquad(G\in\cE_n).
\]
\end{lemma}

\begin{proof}
Positive semidefiniteness gives \(|g_{ij}|\le1\), so \(\cE_n\) is closed and
bounded. If \(G_k\to G\), couple \(X^{G_k}=G_k^{1/2}Z\) and \(X^G=G^{1/2}Z\)
with the same standard Gaussian \(Z\in\R^n\). The positive semidefinite square
root is continuous: the roots \(G_k^{1/2}\) are bounded, and every
subsequential limit is positive semidefinite and squares to \(G\), so it
equals \(G^{1/2}\) by uniqueness of the positive semidefinite root. Thus
\(X^{G_k}\to X^G\) almost surely. If also \(x^k\to x\), the indicators
\(\ind_{\{X^{G_k}\le x^k\}}\) converge to \(\ind_{\{X^G\le x\}}\) outside
\(\bigcup_i\{X^G_i=x_i\}\), a null event because each \(X^G_i\) is standard
Gaussian. Dominated convergence proves continuity. Every row of \(G^{1/2}\)
has norm one, so \(|Z|<t\) implies \(X^G_i<t\) for every \(i\), which gives
the lower bound.
\end{proof}

Recall from \eqref{eq:intro-recurrence} that \(\rho_n(t)=t\sqrt{n/(n-2)}\)
for \(n\ge3\). The following proposition is the only property of minimizing covariances that
the induction uses. Its proof is completed in \cref{sec:interface-end}.

\begin{proposition}\label{prop:interface}
Let \(n\ge3\) and \(t>0\), and suppose \(G\) minimizes \(\Gamma\mapsto F_\Gamma(t)\) on \(\cE_n\). If \(p_0>0\) satisfies
\begin{equation}\label{eq:lower-reference}
 F_\Gamma(\rho_n(t))\ge p_0\qquad\text{for every }\Gamma\in\cE_{n-1},
\end{equation}
then \(F_G\) is differentiable at \(t\) and
\begin{equation}\label{eq:interface}
 F'_G(t)\ge n\varphi(t)p_0.
\end{equation}
If equality holds in \eqref{eq:interface}, then \(G=\Delta_n\).
\end{proposition}

For the simplex, \eqref{eq:interface} holds with equality and
\(p_0=D_{n-1}(\rho_n(t))\). This is the classical recurrence
\eqref{eq:recurrence} for the largest
normalized deviation from a sample mean
\cite[p.~469, (29)]{McKay1935}, \cite[p.~120, (22)]{Nair1948}; Balakrishnan
derives it in our normalization by differentiating through a conditioned
coordinate \cite[pp.~494--495, (4.1)--(4.3)]{Balakrishnan1961}. We include
the short calculation because it fixes the normalization used in the
induction.

\begin{lemma}\label{lem:recurrence}
For \(n\ge3\) and \(t>0\),
\begin{equation}\label{eq:recurrence}
 D'_n(t)=n\varphi(t)D_{n-1}(\rho_n(t)).
\end{equation}
Also \(D_n(t)=0\) for \(t\le0\), and \(D_2(t)=2\Phi(t)-1\) for \(t>0\).
\end{lemma}

\begin{proof}
The vanishing follows from \(\sum_i\xi_i=0\): for \(t<0\) the event
\(\{\xi\le t\one\}\) is empty, and for \(t=0\) it forces every coordinate to
vanish, which has probability zero. For \(n=2\) the simplex vector is
\((Z,-Z)\) with \(Z\) standard Gaussian, so \(D_2(t)=\PP\{|Z|\le t\}\).

For \(n\ge3\), condition one simplex coordinate \(\xi_i\) to equal \(t\). The remaining
coordinates have conditional mean \(-t/(n-1)\), and their conditional
variances and covariances are
\[
 1-\frac1{(n-1)^2}=\frac{n(n-2)}{(n-1)^2}
 \qquad\text{and}\qquad
 -\frac1{n-1}-\frac1{(n-1)^2}=-\frac n{(n-1)^2},
\]
respectively. Their conditional correlations are therefore
\(-1/(n-2)\), the off-diagonal entry of \(\Delta_{n-1}\). Standardizing,
\[
 \PP\{\xi_j\le t\ (j\ne i)\mid\xi_i=t\}
 =D_{n-1}\Bigl(\frac{t+t/(n-1)}{\sqrt{n(n-2)}/(n-1)}\Bigr)
 =D_{n-1}(\rho_n(t)).
\]
The identity \eqref{eq:intro-derivative}, proved in
\cref{prop:boundary-calculus} without using \cref{prop:interface}, now gives
\eqref{eq:recurrence}.
\end{proof}

Now we combine \cref{lem:compact,prop:interface,lem:recurrence}.

\begin{proof}[Proof of \cref{thm:main} from \cref{prop:interface}]
For \(t\le0\) the inequality holds because \(D_n(t)=0\). We proceed by
induction on \(n\).

For \(n=2\) and \(t>0\), inclusion--exclusion gives
\[
 F_G(t)=2\Phi(t)-1+\PP\{X^G_1>t,\ X^G_2>t\}\ge D_2(t).
\]
The last probability is positive if \(-1<g_{12}<1\), since the bivariate
density is positive, and it equals \(1-\Phi(t)>0\) if \(g_{12}=1\). It
vanishes when \(g_{12}=-1\), that is, when \(G=\Delta_2\). This proves the
theorem for \(n=2\).

Now let \(n\ge3\) and assume \eqref{eq:main} with \(n-1\) in place of \(n\). Suppose,
seeking a contradiction, that \(F_{G_0}(t_0)-D_n(t_0)=-\tau<0\) for some
\(G_0\in\cE_n\) and \(t_0>0\). We tilt the difference by a small linear term so
that it attains a negative minimum at an interior threshold. Set \(\varepsilon=\tau/(2t_0)\) and
\[
 \Psi(G,t)=F_G(t)-D_n(t)+\varepsilon t,\qquad t>0.
\]
Then \(\Psi(G_0,t_0)=-\tau/2\), and uniformly in \(G\),
\[
 \Psi(G,t)\ge-D_n(t)+\varepsilon t\longrightarrow0\quad(t\downarrow0),
 \qquad
 \Psi(G,t)\ge-1+\varepsilon t\longrightarrow\infty\quad(t\to\infty).
\]
Choose \(0<t_1<t_0<t_2\) so that \(\Psi>-\tau/2\) for \(t\le t_1\) and
\(t\ge t_2\). By \cref{lem:compact}, \(\Psi\) attains its minimum over
\(\cE_n\times[t_1,t_2]\), and this minimum is at most \(-\tau/2\), so it is
attained at some \((G_*,t_*)\) with \(t_1<t_*<t_2\). Holding \(t_*\) fixed shows
that \(G_*\) minimizes \(G\mapsto F_G(t_*)\).

By the induction hypothesis, \eqref{eq:lower-reference} holds with
\(p_0=D_{n-1}(\rho_n(t_*))\), and \(p_0>0\) by \cref{lem:compact}. Thus
\cref{prop:interface,lem:recurrence} give \(F'_{G_*}(t_*)\ge D'_n(t_*)\). On
the other hand, \(t\mapsto\Psi(G_*,t)\) is minimized at the interior point
\(t_*\) and is differentiable there, so
\[
 0=\partial_t\Psi(G_*,t_*)=F'_{G_*}(t_*)-D'_n(t_*)+\varepsilon
 \ge\varepsilon>0,
\]
a contradiction. This proves \eqref{eq:main} for \(n\).

Finally, suppose \(F_G(t)=D_n(t)\) at some \(t>0\). By
\eqref{eq:main}, \(G\) minimizes \(\Gamma\mapsto F_\Gamma(t)\), and the nonnegative
function \(F_G-D_n\) has a minimum at \(t\). It is differentiable there by
\cref{prop:interface,lem:recurrence}, so
\(F'_G(t)=D'_n(t)=n\varphi(t)D_{n-1}(\rho_n(t))\). This is equality in
\eqref{eq:interface} with \(p_0=D_{n-1}(\rho_n(t))\), and
\cref{prop:interface} gives \(G=\Delta_n\). The converse is immediate.
\end{proof}

The induction uses only the inequality \eqref{eq:main} in dimension \(n-1\).
The equality case in dimension \(n\) comes instead from the equality clause
of \cref{prop:interface}, which applies at each fixed \(t>0\). This is why
\eqref{eq:main-equality} holds at a single threshold. The next three sections
prove \cref{prop:interface}.

\section{Derivatives at the threshold}\label{sec:boundary}

We would like to differentiate \(f_G\) twice at \(t\one\).
\Cref{prop:boundary-calculus} gives
\[
 \partial_if_G(t\one)=\varphi(t)H_i(t),\qquad
 \partial_{ij}f_G(t\one)=q_{ij}\quad(i\ne j),
\]
with \(H_i(t)\) the conditional masses of \eqref{eq:intro-derivative} and
\(q_{ij}\) the pair weights \eqref{eq:q}. The same \(q_{ij}\) give the
covariance derivative in \cref{prop:noise}. The only subtlety is regularity.
The covariance may be singular, and the polyhedron
\(\{y:\ip{v_i}{y}\le t\}\) of the realization below may be unbounded, with
redundant or parallel constraints, so we prove differentiability directly.

Throughout this section \(n\ge3\), \(t>0\), and \(G\in\cE_n\) has distinct
coordinates:
\begin{equation}\label{eq:distinct}
 g_{ij}<1\qquad(i\ne j).
\end{equation}
We fix an intrinsic realization
\(X_i=\ip{v_i}{Z}\), where \(v_1,\ldots,v_n\) are unit vectors spanning
\(\R^r\), \(r=\rank G\), and \(Z\sim\gamma_r\). A rank-one correlation matrix
has only two distinct normals \(\pm v\), so \eqref{eq:distinct} and
\(n\ge3\) force \(r\ge2\).

\subsection{Conditioning and the pair weights}

Conditioning on \(X_i=x_i\) amounts to writing \(Z=x_iv_i+Y\), where \(Y\) is
a standard Gaussian vector in \(v_i^\perp\); this is because the orthogonal
components of a standard Gaussian are independent. Put
\[
 J_i=\{j\ne i:-1<g_{ij}<1\},\qquad \sigma_{ij}=\sqrt{1-g_{ij}^2},
\]
and for \(j\in J_i\) set
\begin{equation}\label{eq:conditional-data}
 w_{ij}=\frac{v_j-g_{ij}v_i}{\sigma_{ij}},\qquad
 \beta_{ij}=t\sqrt{\frac{1-g_{ij}}{1+g_{ij}}}=\frac{t(1-g_{ij})}{\sigma_{ij}}.
\end{equation}
The vectors \(w_{ij}\) are unit vectors in \(v_i^\perp\). Given \(X_i=x_i\),
\(X_j\le x_j\) if and only if
\(\ip{w_{ij}}{Y}\le(x_j-g_{ij}x_i)/\sigma_{ij}\), and at \(x=t\one\) this
offset is \(\beta_{ij}>0\). If \(g_{ij}=-1\), the constraint \(X_j\le x_j\)
becomes \(-x_i\le x_j\), which holds strictly near \(t\one\); we omit it. For
a vector of offsets \(b=(b_j)_{j\in J_i}\in\R^{J_i}\), the conditional region
and its mass are
\begin{equation}\label{eq:conditional-mass}
 Q_i(b)=\{y\in v_i^\perp:\ip{w_{ij}}{y}\le b_j\ (j\in J_i)\},\qquad
 h_i(b)=\gamma_{r-1}(Q_i(b)).
\end{equation}
At the threshold \(t\) the offsets are \(\beta_i=(\beta_{ij})_{j\in J_i}\), so
\[
 H_i=H_i(t)=h_i(\beta_i)=\PP\{X_j\le t\ (j\ne i)\mid X_i=t\}.
\]
Throughout, the index \(i\) of \(Q_i\), \(h_i\), \(\beta_i\), and \(H_i\) is the
conditioned coordinate, and \(j\) indexes the offsets. The standardized
offsets \(\beta_{ij}\) and the formula for \(F'_G\) below appear in
\cite[p.~494, (4.1)]{Balakrishnan1961}.

Conditioning on two coordinates works the same way. For \(-1<g_{ij}<1\),
conditioning on \(X_i=X_j=t\) amounts to writing
\begin{equation}\label{eq:pair-decomposition}
 Z=tv_i+\beta_{ij}w_{ij}+Y',
\end{equation}
where \(Y'\) is a standard Gaussian vector in
\(\operatorname{span}\{v_i,v_j\}^\perp\), which may have dimension zero.
Indeed \(\ip{v_j}{tv_i+\beta_{ij}w_{ij}}=tg_{ij}+\sigma_{ij}\beta_{ij}=t\). Since
\(\beta_{ij}/\sigma_{ij}=t/(1+g_{ij})\), in particular
\begin{equation}\label{eq:pair-mean}
 \E[X_\ell\mid X_i=X_j=t]=t\,\frac{g_{i\ell}+g_{j\ell}}{1+g_{ij}}
 \qquad(\ell\ne i,j).
\end{equation}
For such pairs we define the \emph{pair weights}
\begin{equation}\label{eq:q}
 \begin{split}
 q_{ij}&=\varphi_2(t,t;g_{ij})\,
       \PP\{X_\ell\le t\ (\ell\ne i,j)\mid X_i=X_j=t\},\\
 \varphi_2(t,t;c)&=\frac1{2\pi\sqrt{1-c^2}}\exp\!\left(-\frac{t^2}{1+c}\right),
 \end{split}
\end{equation}
where \(\varphi_2(t,t;c)\) is the standard bivariate Gaussian density with
correlation \(c\), evaluated at \((t,t)\). Set \(q_{ij}=0\) if \(g_{ij}=-1\), and
\(q_{ii}=0\). Then \(q_{ij}=q_{ji}\ge0\).

\subsection{Differentiability}

For bounded polyhedra with positive offsets and distinct prescribed normals,
the offset derivative \eqref{eq:offset-derivative} below follows from the
general-volume formula
of Gardner, Hug, Weil, Xing, and Ye \cite[Theorem~5.3]{GHWXY2019} and its
finite-support form \cite[(22), (25)--(28), and (36)]{GHXY2020}, in dimension
at least two; Huang, Xi, and Zhao state the Gaussian case explicitly
\cite[Theorem~3.3]{HXZ2020}. Our conditional regions can be unbounded and can
lie in dimension \(r-1=1\), so we prove \eqref{eq:offset-derivative} directly.

\begin{lemma}\label{lem:offsets}
Let \(w_1,\ldots,w_m\) be unit vectors in \(\R^d\), \(d\ge1\), and put
\[
 h(b)=\gamma_d\{y:\ip{w_j}{y}\le b_j,\ 1\le j\le m\}.
\]
On the open set of \(b\) for which no two of the hyperplanes
\(\ip{w_j}{y}=b_j\) coincide, \(h\) is \(C^1\), with
\begin{equation}\label{eq:offset-derivative}
 \partial_jh(b)=\varphi(b_j)\,
 \PP\{\ip{w_k}{W}\le b_k\ (k\ne j)\mid \ip{w_j}{W}=b_j\},
 \qquad W\sim\gamma_d.
\end{equation}
\end{lemma}

\begin{proof}
Write \(W=Tw_j+W'\), where \(T\sim\cN(0,1)\) and \(W'\) is a standard Gaussian
vector in \(w_j^\perp\), independent of \(T\). For \(k\ne j\) put
\(c_k=\ip{w_k}{w_j}\) and \(w_k'=w_k-c_kw_j\). Then
\[
 h(b)=\int_{-\infty}^{b_j}\varphi(s)\,R_j(s,b_{-j})\,ds,\qquad
 R_j(s,b_{-j})=\PP\{\ip{w_k'}{W'}\le b_k-c_ks\ (k\ne j)\}.
\]
If \(w_k'\ne0\), the scalar Gaussian \(\ip{w_k'}{W'}\) has no atom at any
threshold. If \(w_k'=0\), then \(w_k=\pm w_j\), and since the two hyperplanes
do not coincide, \(b_k-c_kb_j\ne0\); the corresponding deterministic inequality
has a locally constant truth value. Coupling through the same \(W'\), the
indicators in \(R_j\) therefore converge almost surely as
\((s,b_{-j})\to(b_j,b_{-j})\), and dominated convergence shows that \(R_j\) is
continuous there. The fundamental theorem of calculus gives
\eqref{eq:offset-derivative}, and the same argument with \(b\) varying shows
that this derivative is continuous. The excluded coincidences are finitely
many closed conditions, so their complement is open.
\end{proof}

We apply \cref{lem:offsets} twice. First we apply it in \(\R^r\) to the
normals \(v_1,\ldots,v_n\) with offsets \(x\) near \(t\one\), so that \(h\) is
\(f_G\) and the lemma's \(j\) is our \(i\). The hyperplanes
\(\ip{v_i}{y}=t\) are distinct by \eqref{eq:distinct}. Then, for each \(i\),
we apply it in \(v_i^\perp\) to the normals \(w_{ij}\), \(j\in J_i\), with
offsets \(b\) near \(\beta_i\), so that \(h\) is \(h_i\). The hyperplanes
\(\ip{w_{ij}}{y}=\beta_{ij}\), \(j\in J_i\), are distinct too. Indeed, if
\(w_{i\ell}=-w_{ij}\), the two hyperplanes differ because
\(\beta_{ij},\beta_{i\ell}>0\). If \(w_{ij}=w_{i\ell}\) and
\(\beta_{ij}=\beta_{i\ell}\), then \(g_{ij}=g_{i\ell}\), because
\(c\mapsto t\sqrt{(1-c)/(1+c)}\) is strictly decreasing on \((-1,1)\), and so
\(v_j=g_{ij}v_i+\sigma_{ij}w_{ij}=v_\ell\), contradicting \eqref{eq:distinct}.
Noncoincidence persists in a neighborhood of \(\beta_i\).

\Cref{sec:mass} also evaluates \(h_i\) at the equal offsets
\(\rho_n(t)\one\). There two conditional hyperplanes may coincide, so we
differentiate \(h_i\) only at \(\beta_i\). For example, let \(X=(Z_1,Z_2,(Z_1+Z_2)/\sqrt2)\) with \(Z_1,Z_2\) independent
standard Gaussians. Conditional on \(X_1=t>0\),
\[
 h_1(b)=\Phi(\min\{b_2,b_3\})\quad(b=(b_2,b_3)),\qquad
 \beta_1=(\beta_{12},\beta_{13})=(t,(\sqrt2-1)t).
\]
This mass is differentiable at \(\beta_1\) but not at equal offsets.

We can now compute the derivatives of \(f_G\) at \(t\one\).

\begin{proposition}\label{prop:boundary-calculus}
Under \eqref{eq:distinct}, \(f_G\) is \(C^2\) on a neighborhood of \(t\one\),
and at \(t\one\)
\begin{equation}\label{eq:threshold-derivatives}
 \begin{aligned}
 \partial_if_G&=\varphi(t)H_i,\\
 \partial_{ij}f_G&=q_{ij}&&(i\ne j),\\
 \partial_{ii}f_G&=-t\varphi(t)H_i-\sum_{j\ne i}g_{ij}q_{ij},\\
 F'_G(t)&=\varphi(t)\sum_iH_i.
 \end{aligned}
\end{equation}
The offset derivatives of the conditional masses involve the same weights:
\begin{equation}\label{eq:compatible}
 \partial_jh_i(\beta_i)=\frac{\sigma_{ij}}{\varphi(t)}\,q_{ij}
 \qquad(j\in J_i).
\end{equation}
\end{proposition}

\begin{proof}
By \cref{lem:offsets} applied to the original normals, in a neighborhood of
\(t\one\) in which every antipodal inequality stays strict,
\[
 \partial_if_G(x)=\varphi(x_i)\,
 h_i\Bigl(\Bigl(\frac{x_j-g_{ij}x_i}{\sigma_{ij}}\Bigr)_{j\in J_i}\Bigr).
\]
By \cref{lem:offsets} applied to the conditional normals, the right side is
\(C^1\), so \(f_G\) is \(C^2\) there. At \(x=t\one\) the argument of \(h_i\)
is \(\beta_i\), which gives the first formula.

Applying \cref{lem:offsets} in \(v_i^\perp\) to the \(j\)-th constraint of
\(Q_i\) gives
\[
 \partial_jh_i(\beta_i)=\varphi(\beta_{ij})\,
 \PP\{X_\ell\le t\ (\ell\ne i,j)\mid X_i=X_j=t\}.
\]
Here conditioning \(Y\) on \(\ip{w_{ij}}{Y}=\beta_{ij}\) gives the decomposition
\eqref{eq:pair-decomposition}, so this is the conditional probability in
\eqref{eq:q}. Since \(\beta_{ij}^2=t^2(1-g_{ij})/(1+g_{ij})\),
\[
 \frac{\varphi(t)\varphi(\beta_{ij})}{\sigma_{ij}}
 =\frac1{2\pi\sqrt{1-g_{ij}^2}}
  \exp\!\left[-\frac{t^2}2-\frac{t^2(1-g_{ij})}{2(1+g_{ij})}\right]
 =\varphi_2(t,t;g_{ij}),
\]
which proves \eqref{eq:compatible}. For \(j\in J_i\), only the \(j\)-th
argument of \(h_i\) depends on \(x_j\), with derivative \(1/\sigma_{ij}\), so
differentiating \(\partial_if_G\) in \(x_j\) gives
\(\partial_{ij}f_G=\varphi(t)\,\partial_jh_i(\beta_i)/\sigma_{ij}=q_{ij}\); for
an antipodal \(j\) the omitted inequality contributes zero. Finally,
\(\varphi'(t)=-t\varphi(t)\), and the \(j\)-th argument of \(h_i\) has
\(x_i\)-derivative \(-g_{ij}/\sigma_{ij}\), so
\[
 \partial_{ii}f_G(t\one)
 =-t\varphi(t)H_i+\varphi(t)\sum_{j\in J_i}
  \frac{\sigma_{ij}q_{ij}}{\varphi(t)}\cdot\frac{-g_{ij}}{\sigma_{ij}}
 =-t\varphi(t)H_i-\sum_{j\ne i}g_{ij}q_{ij}.
\]
Summing the first derivatives along \(t\mapsto t\one\) gives the formula for
\(F'_G\).
\end{proof}

We will need one pair weight in every row to be positive, for \(k_i>0\) after
\eqref{eq:row-data} and for the equality case in \cref{sec:mass}. A pair
weight can vanish even for a nonantipodal pair: for
\(X=(Z_1,Z_2,(Z_1+Z_2)/\sqrt2)\), \(X_1=X_2=t\) forces \(X_3=\sqrt2t>t\), so
\(q_{12}=0\).

\begin{lemma}\label{lem:row-positive}
For each \(i\), every index \(j\ne i\) attaining \(\max_{\ell\ne i}g_{i\ell}\)
satisfies \(-1<g_{ij}<1\) and \(q_{ij}>0\).
\end{lemma}

\begin{proof}
Let \(c=g_{ij}\) be such a maximum. If \(c=-1\), every other normal equals
\(-v_i\), and since \(n\ge3\) two of them coincide, contradicting
\eqref{eq:distinct}. Thus \(-1<c<1\). For \(\ell\ne i,j\),
\[
 g_{i\ell}+g_{j\ell}\le c+g_{j\ell}<1+c,
\]
so by \eqref{eq:pair-mean} the conditional mean of \(X_\ell\) is strictly
less than \(t\). By \eqref{eq:pair-decomposition}, conditionally on
\(X_i=X_j=t\) we have \(X_\ell=\E[X_\ell\mid X_i=X_j=t]+\ip{v_\ell}{Y'}\).
Put \(\delta_\ell=t-\E[X_\ell\mid X_i=X_j=t]>0\). Since
\(|\ip{v_\ell}{Y'}|\le|Y'|\), the event
\[
 |Y'|<\min_{\ell\ne i,j}\delta_\ell
\]
has positive probability and satisfies every inequality; when \(Y'\) has
dimension zero, the inequalities hold deterministically. The density
factor in \eqref{eq:q} is positive, so \(q_{ij}>0\).
\end{proof}

With the derivatives in hand, we turn to the covariance.

\section{First-order conditions at a minimizing covariance}\label{sec:variation}

We now use minimality. Adding independent Gaussian noise with covariance
\(sL\), \(L\succeq0\), to \(X^G\) and renormalizing moves \(G\) along a path
\(G_s\) in \(\cE_n\), even when \(G\) is singular, and \cref{prop:noise} gives
\(\frac{d}{ds}F_{G_s}(t)|_{0+}=\frac12\tr(SL)\), where \(S\) is the matrix
\eqref{eq:S} of pair weights. At a minimizer this yields \(S\succeq0\) and
\(SG=0\) (\cref{prop:minimizer}), and then the inequality
\eqref{eq:row-bound} for each row of \(G\) (\cref{lem:row-bound}).

For a fixed symmetric matrix \(L\succeq0\) define
\begin{equation}\label{eq:noise-path}
 R_s=\diag\bigl((1+sL_{ii})^{-1/2}\bigr),\qquad
 G_s=R_s(G+sL)R_s,\qquad s\ge0.
\end{equation}
If \(Y\sim\cN(0,L)\) is independent of \(X^G\), then
\(R_s(X^G+\sqrt sY)\) has covariance \(G_s\), so \(G_s\in\cE_n\). Define the
symmetric matrix
\begin{equation}\label{eq:S}
 S_{ij}=q_{ij}\ (i\ne j),\qquad S_{ii}=-\sum_{j\ne i}g_{ij}q_{ij}.
\end{equation}

\begin{proposition}\label{prop:noise}
Let \(n\ge3\), \(t>0\), and suppose \eqref{eq:distinct} holds. For every fixed
\(L\succeq0\), the right derivative along \eqref{eq:noise-path} is
\begin{equation}\label{eq:noise-derivative}
 \left.\frac{d}{ds}F_{G_s}(t)\right|_{0+}
 =\frac12\tr(SL)
 =\sum_{i<j}q_{ij}\left(L_{ij}-\frac{g_{ij}}2(L_{ii}+L_{jj})\right).
\end{equation}
\end{proposition}

\begin{proof}
Conditioning on \(Y\) gives
\[
 F_{G_s}(t)=\E f_G\bigl(t\one+\zeta(s)-\sqrt sY\bigr),\qquad
 \zeta_i(s)=t\bigl(\sqrt{1+sL_{ii}}-1\bigr)=\frac{st}2L_{ii}+O(s^2).
\]
Put \(u_s=\zeta(s)-\sqrt sY\). Its moments satisfy
\[
 \E u_{s,i}=\frac{st}2L_{ii}+O(s^2),\qquad \E[u_su_s^\top]=sL+O(s^2).
\]
By \cref{prop:boundary-calculus}, Taylor's theorem at \(t\one\) gives
\[
 f_G(t\one+u)=f_G(t\one)+Df_G(t\one)u+\frac12u^\top D^2f_G(t\one)u
 +\mathcal R(u),\qquad \mathcal R(u)=o(|u|^2).
\]
We need \(\E\mathcal R(u_s)=o(s)\). Given \(\theta>0\), choose \(\delta>0\) with
\(|\mathcal R(u)|\le\theta|u|^2\) for \(|u|\le\delta\); on this event the
contribution is at most \(\theta(s\tr L+O(s^2))\). Outside it, define
\(\mathcal R\) by the same subtraction; boundedness of \(f_G\) gives
\(|\mathcal R(u)|\le c_1(1+|u|^2)\). Write \(Y=L^{1/2}Z\) with
\(Z\sim\gamma_n\). If \(L\ne0\), then for all small \(s\),
\[
 |u_s|>\delta\quad\Longrightarrow\quad
 |Z|>\frac\delta{2\norm{L^{1/2}}_{\mathrm{op}}\sqrt s},
\]
and the Gaussian exponential moment with the Cauchy--Schwarz inequality gives
constants \(c_2,c_3>0\) with
\[
 \E\bigl[(1+|u_s|^2)\ind_{\{|u_s|>\delta\}}\bigr]\le c_2e^{-c_3/s}=o(s).
\]
If \(L=0\), the outside event is empty for small \(s\). Dividing by \(s\),
taking the upper limit, and then letting \(\theta\downarrow0\) proves
\(\E\mathcal R(u_s)=o(s)\). Consequently
\[
 \left.\frac{d}{ds}F_{G_s}(t)\right|_{0+}
 =\frac12\sum_{i,j}L_{ij}\,\partial_{ij}f_G(t\one)
  +\frac t2\sum_iL_{ii}\,\partial_if_G(t\one).
\]
Inserting \eqref{eq:threshold-derivatives}, the terms \(t\varphi(t)H_iL_{ii}\)
cancel, and what remains is \(\frac12\tr(SL)\). Expanding the trace gives the
last expression in \eqref{eq:noise-derivative}.
\end{proof}

Covariance derivatives of Gaussian probabilities go back to Plackett
\cite[pp.~352--353, (3)--(6)]{Plackett1954}, with a
distributional formulation in \cite[Theorem~1]{Voigtlaender2020}; those
formulas assume a positive definite covariance, which is why we proved
\cref{prop:noise} at a possibly singular \(G\).

We use the following row data of \(S\):
\begin{equation}\label{eq:row-data}
 \begin{split}
 k_i&=\sum_{j\ne i}(1-g_{ij})q_{ij}=(S\one)_i,\qquad
 C_i=\sum_{j\ne i}\sqrt{1-g_{ij}^2}\,q_{ij},\\
 \kappa&=\sum_ik_i=\one^\top S\one,\qquad a_i=\frac{k_i}\kappa.
 \end{split}
\end{equation}
By \eqref{eq:two-derivatives} below,
\(\beta_i\cdot\nabla h_i(\beta_i)=tk_i/\varphi(t)\) and
\(\one\cdot\nabla h_i(\beta_i)=C_i/\varphi(t)\). By \cref{lem:row-positive},
\(k_i>0\) and \(C_i>0\), so \(\kappa>0\), \(a_i>0\), and \(\sum_ia_i=1\).

The row estimate in \cref{lem:row-bound} needs a lower bound on \(S_{ii}\)
in terms of \(k_i\) and \(\kappa\). Let \(\varepsilon_i\) be the \(i\)-th
coordinate vector. When \(S\succeq0\), testing it on
\(\varepsilon_i-c\one\) with \(c\in\R\) gives
\[
 0\le(\varepsilon_i-c\one)^\top S(\varepsilon_i-c\one)
 =S_{ii}-2ck_i+c^2\kappa
 =S_{ii}-\frac{k_i^2}\kappa+\kappa\Bigl(c-\frac{k_i}\kappa\Bigr)^2,
\]
which is strongest at \(c=a_i\). We therefore put
\begin{equation}\label{eq:tests}
 z_i=\varepsilon_i-a_i\one,\qquad
 \eta_i=z_i^\top Sz_i=S_{ii}-\frac{k_i^2}\kappa.
\end{equation}

\begin{proposition}\label{prop:minimizer}
Let \(n\ge3\) and \(t>0\). Every global minimizer \(G\) of \(\Gamma\mapsto F_\Gamma(t)\)
on \(\cE_n\) satisfies \eqref{eq:distinct}, and with \(k=S\one\),
\[
 S\succeq0,\qquad SG=0,\qquad Gk=0,\qquad k_i>0,
 \qquad 2\le\rank G\le n-1.
\]
In particular \(\eta_i\ge0\) for every \(i\).
\end{proposition}

\begin{proof}
A minimizer exists by \cref{lem:compact}. Suppose \(g_{ij}=1\), so that
\(X_i=X_j\) almost surely. Replacing \(X_i\) by an independent standard
Gaussian gives \(G'\in\cE_n\) with \(F_{G'}(t)=\Phi(t)F_G(t)\), because the
old \(i\)-th constraint was redundant. Since \(F_G(t)>0\)
(\cref{lem:compact}), \(F_{G'}(t)<F_G(t)\), a contradiction. Thus
\eqref{eq:distinct} holds.

For every \(z\in\R^n\), \(L=zz^\top\) is admissible in \cref{prop:noise}, and
minimality gives
\[
 0\le\left.\frac{d}{ds}F_{G_s}(t)\right|_{0+}=\frac12z^\top Sz.
\]
Thus \(S\succeq0\). Also, directly from \eqref{eq:S} and
\(g_{ii}=1\),
\[
 \tr(SG)=\sum_iS_{ii}+\sum_{i\ne j}g_{ij}q_{ij}=0.
\]
This identity holds for every \(G\) satisfying \eqref{eq:distinct}: it is
\cref{prop:noise} with \(L=G\), for which \(G_s=G\) for every \(s\). Since
\(S\) and \(G\) are positive semidefinite,
\(\tr(SG)=\norm{G^{1/2}S^{1/2}}_F^2\), so \(G^{1/2}S^{1/2}=0\) and
\(SG=GS=0\). In particular \(Gk=GS\one=0\), where \(k_i>0\) as noted after
\eqref{eq:row-data}. A nonzero kernel vector gives \(\rank G\le n-1\), and
\(\rank G\ge2\) was observed at the start of \cref{sec:boundary}. Finally,
\(\eta_i\ge0\) is the case \(z=z_i\) of \(z^\top Sz\ge0\), with \(z_i\) from
\eqref{eq:tests}.
\end{proof}

To see what the row inequalities should say, we compute the row data at the simplex. At \(G=\Delta_n\), symmetry and \cref{lem:row-positive} give \(q_{ij}=q>0\) for
all \(i\ne j\), and then
\[
  S=qJ,\quad k_i=nq,\quad \kappa=n^2q,\quad a_i=\frac1n,\quad \eta_i=0,\quad C_i=\sqrt{n(n-2)}\,q,\quad \beta_{ij}=\rho_n(t).
\]
Thus \(C_i/k_i=\sqrt{(n-2)/n}\) at the simplex, and \cref{lem:row-bound}
compares \(A_i=C_i/k_i\) with this value \(A_*\).

\begin{lemma}\label{lem:row-bound}
Let \(n\ge3\), \(t>0\), and \(G\in\cE_n\). Suppose \eqref{eq:distinct} holds
and \(\eta_i\ge0\) for every \(i\). Put
\[
 A_*=\sqrt{\frac{n-2}n},\qquad A_i=\frac{C_i}{k_i}.
\]
Then \(0<a_i<1/2\) and
\begin{equation}\label{eq:row-bound}
 C_i^2\le k_i^2(1-2a_i),\qquad
 (n-2)\Bigl(1-\frac{A_i}{A_*}\Bigr)\ge na_i-1.
\end{equation}
\end{lemma}

\begin{proof}
The Cauchy--Schwarz inequality applied to the lists
\(\sqrt{(1-g_{ij})q_{ij}}\) and \(\sqrt{(1+g_{ij})q_{ij}}\) gives
\begin{equation}\label{eq:row-cauchy}
 \begin{split}
 C_i^2&\le k_i\sum_{j\ne i}(1+g_{ij})q_{ij}=k_i(k_i-2S_{ii})\\
 &\le k_i\Bigl(k_i-\frac{2k_i^2}\kappa\Bigr)=k_i^2(1-2a_i),
 \end{split}
\end{equation}
where the equality uses \(\sum_{j\ne i}(1+g_{ij})q_{ij}
=\sum_{j\ne i}(1-g_{ij})q_{ij}+2\sum_{j\ne i}g_{ij}q_{ij}=k_i-2S_{ii}\), and
the second inequality is \(\eta_i\ge0\). Because \(C_i>0\), this forces
\(a_i<1/2\). The tangent line of the concave function
\(a\mapsto\sqrt{1-2a}\) at \(a=1/n\) gives
\[
 A_i\le\sqrt{1-2a_i}\le A_*-\frac{a_i-1/n}{A_*}.
\]
Dividing by \(A_*\), multiplying by \(n-2\), and using \((n-2)/A_*^2=n\)
gives the second inequality in \eqref{eq:row-bound}.
\end{proof}

Since \(\sum_ia_i=1\), the right sides \(na_i-1\) of \eqref{eq:row-bound}
sum to zero, so unless every \(a_i=1/n\) some row has a negative right side,
and there \(A_i\) may exceed \(A_*\). \Cref{sec:mass} shows that the row inequalities together still give
\(\sum_iH_i\ge np_0\) (\cref{prop:finite-boundary}).

\begin{remark}\label{rem:bounded}
The kernel vector has a geometric meaning. In the intrinsic realization,
\(0=k^\top Gk=|\sum_ik_iv_i|^2\), so \(\sum_ik_iv_i=0\) with every \(k_i>0\).
If \(\ip{v_i}{y}\le0\) for every \(i\), then \(\sum_ik_i\ip{v_i}{y}=0\) forces
\(\ip{v_i}{y}=0\) for every \(i\), and \(y=0\) because the \(v_i\) span
\(\R^r\). Thus the polyhedron \(\{y\in\R^r:\ip{v_i}{y}\le t\}\) of a
minimizing covariance is bounded.
\end{remark}

\section{The aggregate bound and its equality case}\label{sec:mass}

We now prove \cref{prop:interface}. Of the conclusions of
\cref{prop:minimizer} we need only \eqref{eq:distinct}, \(\rank G\le n-1\)
and \(\eta_i\ge0\), so \cref{prop:finite-boundary} assumes only these.

\begin{proposition}\label{prop:finite-boundary}
Let \(n\ge3\), \(t>0\), and let \(G\in\cE_n\) satisfy \eqref{eq:distinct}
and \(\rank G\le n-1\). Suppose \(\eta_i\ge0\) for every \(i\), and \(p_0>0\)
satisfies
\[
 h_i(\rho_n(t)\one)\ge p_0\qquad(1\le i\le n).
\]
Then
\[
 \sum_iH_i\ge np_0\qquad\text{and}\qquad F'_G(t)\ge n\varphi(t)p_0.
\]
Equality in either inequality forces \(G=\Delta_n\).
\end{proposition}

The proof compares each \(H_i=h_i(\beta_i)\) with \(h_i(\rho_n(t)\one)\ge p_0\),
using two facts about the conditional masses. Log-concavity in the offsets
gives the tangent inequality \eqref{eq:reciprocal}, and the dilation identity \eqref{eq:dilation} splits
\(H_i=u_i+e_i\) with \(u_i,e_i>0\) \eqref{eq:decomposition}.
\Cref{lem:scalar} sums the resulting inequalities \eqref{eq:components}.

\subsection{Padding and two derivatives}

Put \(r=\rank G\). The row bound \eqref{eq:row-bound} is calibrated to
dimension \(n-2\) through \(A_*=\sqrt{(n-2)/n}\), and the dilation identity
\eqref{eq:dilation} below must be used in the same dimension for the term
\(-u_i\) to cancel in \eqref{eq:components}. So we work in \(\R^{n-2}\), and
this is where \(\rank G\le n-1\) enters: since \(r-1\le n-2\), we adjoin unused
Gaussian coordinates to each \(v_i^\perp\), replace each region \(Q_i(b)\) by
\(Q_i(b)\times\R^{n-1-r}\subset\R^{n-2}\), and keep the names \(Q_i\) and
\(h_i\). The masses \(h_i(b)\) and their offset derivatives do not change.\footnote{Padding changes the dilation identity
\eqref{eq:dilation} only through the second moment: if
\(Q_i'\subset\R^{r-1}\) is the region before padding, so that
\(Q_i=Q_i'\times\R^{n-1-r}\), then
\(\int_{Q_i}|y|^2\,d\gamma_{n-2}(y)
=\int_{Q_i'}|y'|^2\,d\gamma_{r-1}(y')+(n-1-r)H_i\).} By \eqref{eq:conditional-data} and \eqref{eq:compatible}, and because
\(\beta_{ij}\sigma_{ij}=t(1-g_{ij})\),
\begin{equation}\label{eq:two-derivatives}
 \begin{aligned}
 \beta_i\cdot\nabla h_i(\beta_i)
 &=\frac1{\varphi(t)}\sum_{j\in J_i}\beta_{ij}\sigma_{ij}q_{ij}
 =\frac{tk_i}{\varphi(t)},\\
 \one\cdot\nabla h_i(\beta_i)
 &=\frac1{\varphi(t)}\sum_{j\in J_i}\sigma_{ij}q_{ij}
 =\frac{C_i}{\varphi(t)}.
 \end{aligned}
\end{equation}
Since \(q_{ij}=0\) for antipodal pairs, these sums over \(J_i\) equal the sums
over \(j\ne i\) in \eqref{eq:row-data}.

\subsection{The reciprocal tangent and the dilation identity}

The first fact is log-concavity in the offsets. For fixed unit normals \(w_1,\ldots,w_m\) in \(\R^{n-2}\), the function
\[
 (y,b)\longmapsto\varphi_{n-2}(y)\,\ind_{\{\ip{w_j}{y}\le b_j\ \forall j\}}
\]
is log-concave on \(\R^{n-2}\times\R^m\): it is the product of a log-concave
density and the indicator of a convex set. Pr\'ekopa's theorem
\cite[Theorem~6, p.~342]{Prekopa1973} therefore makes its marginal
\(h(b)=\gamma_{n-2}\{y:\ip{w_j}{y}\le b_j\text{ for every }j\}\) log-concave.
Unbounded polyhedra are allowed, since the integrand is measurable and bounded
by the integrable Gaussian density. On the positive orthant of offsets \(h>0\), so \(1/h=e^{-\log
h}\) is convex. At a point \(b\) where \(h\) is differentiable, the tangent
inequality \(1/h(c)\ge1/h(b)-\nabla h(b)\cdot(c-b)/h(b)^2\), multiplied by
\(h(b)^2\), reads
\begin{equation}\label{eq:reciprocal}
 \frac{h(b)^2}{h(c)}-h(b)\ge(b-c)\cdot\nabla h(b).
\end{equation}
Only differentiability at \(b\) is used. Huang, Xi, and Zhao prove the logarithmic form of this tangent inequality for convex bodies, in terms of support functions \cite[Lemma~4.3]{HXZ2020}. We use the reciprocal form because it produces the
quadratic term \(H_i^2/p_0\) that \cref{lem:scalar} below aggregates.

The second fact is an identity rather than an inequality. Since
\(Q_i(s\beta_i)=sQ_i(\beta_i)\) for \(s>0\), a change of variables gives
\[
 h_i(s\beta_i)=\int_{Q_i(\beta_i)}s^{n-2}(2\pi)^{-(n-2)/2}e^{-s^2|y|^2/2}\,dy.
\]
For \(1/2\le s\le3/2\) the \(s\)-derivative of the integrand is bounded by a
constant multiple of the integrable function \((1+|y|^2)e^{-|y|^2/8}\), so we
may differentiate at \(s=1\):
\begin{equation}\label{eq:dilation}
 \beta_i\cdot\nabla h_i(\beta_i)=(n-2)H_i-\int_{Q_i(\beta_i)}|y|^2\,d\gamma_{n-2}(y).
\end{equation}
Comparing \eqref{eq:dilation} with \eqref{eq:two-derivatives}, we split
\begin{equation}\label{eq:decomposition}
 H_i=u_i+e_i,\qquad
 u_i=\frac{tk_i}{(n-2)\varphi(t)},\qquad
 e_i=\frac1{n-2}\int_{Q_i(\beta_i)}|y|^2\,d\gamma_{n-2}(y).
\end{equation}
Here \(u_i>0\) because \(k_i>0\). Also \(J_i\ne\varnothing\) by
\cref{lem:row-positive}, the ball of radius \(\frac12\min_{j\in J_i}\beta_{ij}>0\)
lies in \(Q_i(\beta_i)\), and \(n-2\ge1\), so \(e_i>0\). If \(U=\sum_iu_i\), then
\begin{equation}\label{eq:normalized-u}
 \frac{u_i}U=\frac{k_i}\kappa=a_i.
\end{equation}

\subsection{From row inequalities to an aggregate bound}

Put \(\rho=\rho_n(t)=t/A_*\), and apply \eqref{eq:reciprocal} with \(b=\beta_i\)
and \(c=\rho\one\). The hypothesis \(h_i(\rho\one)\ge p_0\),
\eqref{eq:two-derivatives} and \cref{lem:row-bound} give
\begin{equation}\label{eq:component-chain}
 \begin{aligned}
 \frac{H_i^2}{p_0}-H_i
 &\ge\frac{H_i^2}{h_i(\rho\one)}-H_i
 \ge(\beta_i-\rho\one)\cdot\nabla h_i(\beta_i)\\
 &=\frac{tk_i-\rho C_i}{\varphi(t)}
 =(n-2)u_i\Bigl(1-\frac{A_i}{A_*}\Bigr)
 \ge u_i(na_i-1).
 \end{aligned}
\end{equation}
The right side \(u_i(na_i-1)\) is negative when \(a_i<1/n\), so
\eqref{eq:component-chain} alone does not give \(H_i\ge p_0\).
Substituting \(H_i=u_i+e_i\) and \(a_i=u_i/U\) cancels the term \(-u_i\) and
leaves
\begin{equation}\label{eq:components}
 \frac{H_i^2}{np_0}-\frac{u_i^2}U\ge\frac{e_i}n,
 \qquad H_i=u_i+e_i,\quad u_i,e_i>0.
\end{equation}
Both \eqref{eq:components} and \eqref{eq:decomposition} use
\eqref{eq:two-derivatives}: the same \(k_i=(S\one)_i\) from the first-order
conditions appears in \(\beta_i\cdot\nabla h_i(\beta_i)=tk_i/\varphi(t)\) and in
\(u_i\). The following scalar lemma sums the inequalities
\eqref{eq:components}.

\begin{lemma}\label{lem:scalar}
Let \(u_i,e_i>0\), \(x_i=u_i+e_i\), \(U=\sum_iu_i\), and \(E=\sum_ie_i\).
Let \(M>0\), and let \(\omega_i>0\) satisfy \(\sum_i\omega_i=1\). If
\[
 \frac{x_i^2}M-\frac{u_i^2}U\ge\omega_ie_i\qquad\text{for every }i,
\]
then \(P:=\sum_ix_i\ge M\). If \(P=M\), then for every \(i\)
\[
 \frac{u_i}U=\frac{e_i}E=\omega_i,\qquad x_i=M\omega_i.
\]
\end{lemma}

\begin{proof}
Since \(P=U+E\), the Cauchy--Schwarz inequality gives
\begin{equation}\label{eq:scalar-cauchy}
 \frac{x_i^2}P\le\frac{u_i^2}U+\frac{e_i^2}E.
\end{equation}
Some index has \(e_i/E\le\omega_i\), because both sides sum to one. For this
index,
\[
 \frac{x_i^2}M\ge\frac{u_i^2}U+\omega_ie_i\ge\frac{u_i^2}U+\frac{e_i^2}E
 \ge\frac{x_i^2}P,
\]
and since \(x_i>0\), \(P\ge M\).

If \(P=M\), the hypothesis and \eqref{eq:scalar-cauchy} give, for every \(i\),
\[
 \omega_ie_i\le\frac{x_i^2}P-\frac{u_i^2}U\le\frac{e_i^2}E,
\]
so \(\omega_i\le e_i/E\). Summing forces \(e_i/E=\omega_i\) for every \(i\),
and then equality holds in \eqref{eq:scalar-cauchy}. The identity
\[
 \frac{u_i^2}U+\frac{e_i^2}E-\frac{x_i^2}{U+E}=\frac{(Eu_i-Ue_i)^2}{(U+E)UE}
\]
then gives \(u_i/U=e_i/E=\omega_i\), and \(x_i=(U+E)\omega_i=M\omega_i\).
\end{proof}

Apply \cref{lem:scalar} to \eqref{eq:components} with \(x_i=H_i\),
\(M=np_0\), and \(\omega_i=1/n\). This gives \(\sum_iH_i\ge np_0\), and
\eqref{eq:threshold-derivatives} gives \(F'_G(t)\ge n\varphi(t)p_0\).

\subsection{Equality}

At \(G=\Delta_n\), with \(p_0=D_{n-1}(\rho_n(t))\), the inequalities
\eqref{eq:row-cauchy}, \eqref{eq:component-chain} and \(\sum_iH_i\ge np_0\)
are equalities, since there
\(\beta_i=\rho_n(t)\one\), \(a_i=1/n\), \(A_i=A_*\), and
\(H_i=D_{n-1}(\rho_n(t))\). Conversely, suppose \(\sum_iH_i=np_0\). The
equality case of \cref{lem:scalar} gives
\begin{equation}\label{eq:balanced-equality}
 H_i=p_0,\qquad a_i=\frac1n\qquad(1\le i\le n).
\end{equation}
Then the first and last members of \eqref{eq:component-chain} vanish, so every
inequality in that chain is an equality. Since \(u_i>0\), this gives
\(A_i=A_*=\sqrt{1-2a_i}\), that is, \(C_i^2=k_i^2(1-2a_i)\), so
\eqref{eq:row-cauchy} is an equality throughout.

Equality in its Cauchy--Schwarz step makes the two lists proportional, with
constant \(C_i/k_i=A_*\). Thus, whenever \(q_{ij}>0\),
\[
 1+g_{ij}=A_*^2(1-g_{ij})=\frac{n-2}n(1-g_{ij}),
 \qquad\text{so}\qquad g_{ij}=-\frac1{n-1}.
\]
In each row a maximal off-diagonal entry has a positive weight by
\cref{lem:row-positive}, so every off-diagonal entry satisfies
\(g_{ij}\le-1/(n-1)\). Positive semidefiniteness now gives
\[
 0\le\one^\top G\one=n+2\sum_{i<j}g_{ij}\le n-\frac{n(n-1)}{n-1}=0,
\]
so every entry attains the bound and \(G=\Delta_n\). This completes the proof
of \cref{prop:finite-boundary}.

\subsection{Proof of the derivative bound}\label{sec:interface-end}

\begin{proof}[Proof of \cref{prop:interface}]
By \cref{prop:minimizer}, a minimizing covariance has distinct coordinates,
rank at most \(n-1\), and \(\eta_i\ge0\) for every \(i\). For each \(i\), the
conditional coordinates \((\ip{w_{ij}}{Y})_{j\in J_i}\) are at most \(n-1\) standard Gaussians, and at least one by \cref{lem:row-positive}. Repeating one of them
if necessary gives a vector with a covariance \(\Gamma_i\in\cE_{n-1}\) and
\(h_i(\rho_n(t)\one)=F_{\Gamma_i}(\rho_n(t))\), so \eqref{eq:lower-reference} gives
\(h_i(\rho_n(t)\one)\ge p_0\), and \cref{prop:finite-boundary} gives the
derivative bound and its equality case. Differentiability is part of
\cref{prop:boundary-calculus}.
\end{proof}

This completes the proof of \cref{thm:main}. We turn to its consequences.

\section{Consequences}\label{sec:consequences}

In this section we ask what \cref{thm:main} gives beyond the smoothed
comparison \eqref{eq:predecessor}. \Cref{cor:integrated} gives
\(\E\psi(M_G)\le\E\psi(M_{\Delta_n})\) for nondecreasing \(\psi\). Its cases
\(\psi(x)=x\) and \(\psi(x)=e^{\lambda x}\) already follow from
\eqref{eq:predecessor}, but a threshold indicator does not
(\cref{sec:smoothed}). The main
application, \cref{thm:identification}, is identification with an inactive
state, whose optimal success probability at amplitude \(\lambda>0\) and
false-alarm level \(\alpha\) is the truncated exponential moment
\(e^{-\lambda^2/2}\E[e^{\lambda M_G}\ind_{\{M_G>t_G(\alpha)\}}]/n\).
\Cref{sec:graphs} applies \eqref{eq:predecessor} and \cref{thm:main} to
independent sets in arc-transitive graphs.

\subsection{Integrated comparisons}\label{sec:integrated}

Since \eqref{eq:main} holds at every threshold, it passes to nondecreasing functions of the maximum.

\begin{corollary}\label{cor:integrated}
Let \(G\in\cE_n\), and let \(\psi:\R\to\R\) be nondecreasing with
\(\E|\psi(M_G)|<\infty\) and \(\E|\psi(M_{\Delta_n})|<\infty\). Then
\[
 \E\psi(M_G)\le\E\psi(M_{\Delta_n}).
\]
If \(G\ne\Delta_n\) and \(\psi(x)<\psi(y)\) for some \(0<x<y\), the
inequality is strict.
\end{corollary}

\begin{proof}
With \(U\) uniform on \((0,1)\), the quantile functions give random variables
\(F_G^{-1}(U)\) and \(D_n^{-1}(U)\) with the laws of \(M_G\) and
\(M_{\Delta_n}\), and \eqref{eq:main} gives \(F_G^{-1}(U)\le D_n^{-1}(U)\).
Applying \(\psi\) and taking expectations proves the inequality.

For strictness, note that \(M_G\) has no atoms, since
\(\{M_G=s\}\subset\bigcup_i\{X^G_i=s\}\); the same holds for \(M_{\Delta_n}\).
Replacing \(\psi\) by its right-continuous version changes it at countably
many points only, so it changes neither expectation; we keep the name
\(\psi\) for this version and let \(\mu_\psi\) be its Lebesgue--Stieltjes
measure. For every real \(x\),
\[
 \psi(x)-\psi(0)=\int_{(0,\infty)}\ind_{\{s\le x\}}\,d\mu_\psi(s)
 -\int_{(-\infty,0]}\ind_{\{x<s\}}\,d\mu_\psi(s),
\]
and at most one of the two terms is nonzero, so each has finite expectation
at \(x=M_G\) and at \(x=M_{\Delta_n}\). Taking expectations with Tonelli's
theorem, and using \(\PP\{M\ge s\}=1-\PP\{M\le s\}\) for the atomless maxima,
gives
\begin{equation}\label{eq:integrated-identity}
 \E\psi(M_{\Delta_n})-\E\psi(M_G)=\int_\R\bigl[F_G(s)-D_n(s)\bigr]\,
 d\mu_\psi(s).
\end{equation}
The integrand is nonnegative, and by \eqref{eq:main-equality} it is strictly
positive at every \(s>0\) when \(G\ne\Delta_n\). If \(\psi(x)<\psi(y)\) with
\(0<x<y\), then \(\mu_\psi([x,y])>0\), and the integral is positive.
\end{proof}

For \(n\ge3\) the condition on \(\psi\) is also necessary for strictness at
every \(G\ne\Delta_n\): if \(\psi\) is constant on \((0,\infty)\), a matrix
\(G\) with an antipodal pair has \(F_G=D_n=0\) on \((-\infty,0]\), and
\eqref{eq:integrated-identity} gives equality. For \(n=2\), \(F_G-D_2>0\) on all of
\(\R\) when \(G\ne\Delta_2\), so every nonconstant \(\psi\) gives
strictness.

\Cref{thm:main} also extends to Gaussian scale mixtures. If \(A>0\) is
independent of \(X^G\), conditioning on \(A\) gives
\(\PP\{AM_G\le t\}=\E F_G(t/A)\ge\E D_n(t/A)=\PP\{AM_{\Delta_n}\le t\}\) for
every \(t\), with strict inequality for \(t>0\) when \(G\ne\Delta_n\).

Two cases of \cref{cor:integrated} are classical objectives. Taking \(\psi(x)=x\) gives
\(\E M_G\le\E M_{\Delta_n}=\sqrt{n/(n-1)}\,\E\max_iZ_i\), since
\(\max_i(Z_i-\overline Z)=\max_iZ_i-\overline Z\). For \(Z\sim\gamma_m\),
\[
 \E\max_i\ip{v_i}{Z}=\tfrac12\E|Z|\,w(\conv\{v_1,\ldots,v_n\}),\qquad
 w(K)=2\E\max_{y\in K}\ip{\theta}{y},
\]
where \(\theta\) is uniform on \(S^{m-1}\) and \(w(K)\) is the mean width of
\(K\) \cite[(1), (3), and Proposition~1.1]{KLZ2016}, \cite[p.~74, (1)]{Litvak2018}. So among convex
hulls of \(n\) unit vectors in \(\R^m\), \(m\ge n-1\), the regular simplex
maximizes mean width, uniquely up to an orthogonal transformation. For
\(m=n-1\) this is the Simplex Mean Width Conjecture
\cite[Conjecture~1.1]{Litvak2018}, whose Gaussian form
\cite[Conjecture~1.4]{Litvak2018} is the case \(\psi(x)=x\) above. Litvak
observed that the conjecture of Balitskiy, Karasev, and Tsigler would imply it
\cite[pp.~75--76]{Litvak2018}, and Sun, Hu, and Lan proved the case \(n=4\)
\cite{SunHuLan2020}.

Taking \(\psi(x)=e^{\lambda x}\), \(\lambda>0\), gives the Weak Simplex Conjecture:
for equally likely signals \(\lambda v_1,\ldots,\lambda v_n\) with \(v_i\)
unit vectors in \(\R^m\), sent through additive noise \(Z\sim\gamma_m\),
maximum-likelihood decoding succeeds with probability
\begin{equation}\label{eq:decoding}
 P_{\mathrm c}=\frac1n\int_{\R^m}\max_i\varphi_m(y-\lambda v_i)\,dy
 =\frac{e^{-\lambda^2/2}}n\,\E\exp\Bigl(\lambda\max_i\ip{v_i}{Z}\Bigr),
\end{equation}
as follows by expanding \(|y-\lambda v_i|^2=|y|^2+\lambda^2-2\lambda
\ip{v_i}{y}\) \cite[pp.~487--488, (2.2)--(2.3)]{Balakrishnan1961}.\footnote{The first
expression holds with repeated signals and any measurable tie rule.} So the regular
simplex maximizes \(P_{\mathrm c}\), uniquely when \(m\ge n-1\); when
\(m<n-1\) no configuration in \(\R^m\) has Gram matrix \(\Delta_n\) and the
bound is strict for every code. Both comparisons appear in \cite[Corollaries~2.5--2.7]{Mulgund2026}, where the mean-width statement is proved for points anywhere in the unit ball. Their equality cases were obtained independently in our preceding paper
\cite{Mulgund2026} and by Su et al., for decoding in \cite[Theorems~2 and~3(a)]{SuYangXuI2026}
and for the mean width of full-dimensional simplices in the unit ball in \cite[Corollary~4]{SuYangXuI2026}.

\subsection{Identification with an inactive state}\label{sec:identification}

Ordinary decoding always outputs a label. Suppose instead that the receiver may
declare that no signal is present, and that we constrain how often it raises a
false alarm. For a fixed amplitude \(\lambda>0\), the observation is
\[
 Y=Z\quad\text{under inactivity},\qquad
 Y=\lambda u_i+Z\quad\text{under active label }i,\qquad Z\sim\cN(0,I_m),
\]
where \(u_1,\ldots,u_n\in\R^m\) are unit vectors with Gram matrix \(G\). A
decision rule outputs \(0\) (inactive) or a label. Its false-alarm
probability is the probability that it outputs a label under inactivity.
Among rules with false-alarm probability at most \(\alpha\), we want to
maximize the average probability of outputting the correct label under the
\(n\) equally weighted active hypotheses, and we ask which signal geometry
makes this largest at every \(\alpha\).

The same analysis applies when the amplitude is random but unobserved, with a
common law for all labels, and we state it in that generality. Fix \(n\ge2\),
\(m\ge1\), and unit vectors \(u_1,\ldots,u_n\in\R^m\) with Gram matrix \(G\).
Let the amplitude \(\Lambda\) be nonnegative, finite almost surely, independent of \(Z\), with the same law \(\mu\) under every active label. Write \(\PP_0\) for the law of \(Y=Z\) and
\(\PP_{i,\mu}\) for the law of \(Y=\Lambda u_i+Z\). A decision rule \(\delta\), which
may use independent randomization, takes values in \(\{0,1,\ldots,n\}\); its
false-alarm probability is \(\operatorname{FA}(\delta)=\PP_0\{\delta\ne0\}\).
The best average probability of correct identification is
\begin{equation}\label{eq:id-value-definition}
 C_G(\mu,\alpha)=\sup_{\delta:\operatorname{FA}(\delta)\le\alpha}
 \frac1n\sum_{i=1}^n\PP_{i,\mu}\{\delta=i\},
\end{equation}
and for a deterministic amplitude \(\lambda\ge0\) we write \(C_G(\lambda,\alpha)\) for the value when \(\mu\) is the point mass at \(\lambda\). By \eqref{eq:id-frontier} below, the value depends on the vectors only through \(G\), and not on \(m\), so \(C_G(\mu,\alpha)\) is defined for every \(G\in\cE_n\) through vectors realizing \(G\) in any dimension \(m\ge\rank G\). Under \(\PP_0\) we realize
\(X^G_i=\ip{u_i}{Y}\), and we define
\begin{equation}\label{eq:id-profile}
 \ell_\mu(x)=\int_{[0,\infty)}e^{\lambda x-\lambda^2/2}\,d\mu(\lambda),\qquad
  B_G=\frac1n\ell_\mu(M_G).
\end{equation}
Since the density of \(\cN(\lambda u_i,I_m)\) with respect to \(\cN(0,I_m)\) is \(e^{\lambda\ip{u_i}{y}-\lambda^2/2}\), the function \(y\mapsto\ell_\mu(\ip{u_i}{y})\) is the likelihood ratio of \(\PP_{i,\mu}\) with respect to \(\PP_0\).

For a fixed code, the optimal rule is classical: accept the largest
likelihood ratio when it exceeds a threshold. Weinberger and Merhav give it
with false-alarm and missed-detection constraints
\cite[Section~III, Lemma~1, (11)--(12)]{WeinbergerMerhav2014}, where a missed
detection is an output \(0\) under an active label. Setting the
multiplier of the missed-detection constraint to zero gives the rule used
here; see also \cite[Section~II-B, (4)--(5)]{ObermuellerEtAl2025}. The false-alarm, missed-detection, and inclusive-error criteria are separated in
\cite[Section~II, (3)--(5)]{LanchoOstmanDurisi2021}. An inclusive error is any
output other than the active label, so our objective is one minus the
inclusive-error probability. Here we optimize over the code and identify the optimal
equal-energy geometry, with uniqueness.

\begin{theorem}\label{thm:identification}
Suppose \(\mu((0,\infty))>0\) and \(0<\alpha\le1\). For \(0<\alpha<1\), let
\(t_G(\alpha)\) satisfy
\begin{equation}\label{eq:id-null-quantile}
 \PP\{M_G>t_G(\alpha)\}=\alpha.
\end{equation}
An optimal rule selects a label maximizing \(\ip{u_i}{Y}\) and outputs it if
\(\max_i\ip{u_i}{Y}>t_G(\alpha)\), and outputs \(0\) otherwise; at
\(\alpha=1\) it always outputs a maximizing label. The same rule is optimal for every such amplitude law, and in particular for every deterministic amplitude \(\lambda>0\). The optimal value satisfies
\begin{align}
 C_G(\mu,\alpha)&=\int_0^\infty\min\{\alpha,\PP(B_G>s)\}\,ds
 \le C_{\Delta_n}(\mu,\alpha),\label{eq:id-frontier}\\
 C_G(\mu,\alpha)&=\int_{[0,\infty)}C_G(\lambda,\alpha)\,d\mu(\lambda),\label{eq:id-linearity}
\end{align}
and equality holds in the comparison in \eqref{eq:id-frontier} if and only if
\(G=\Delta_n\). If \(m\ge n-1\), a regular simplex attains the bound, and it
also maximizes the smallest success probability over labels:
\begin{equation}\label{eq:id-label-minimax}
 \sup_{u_1,\ldots,u_n\in S^{m-1}}\ \sup_{\delta:\operatorname{FA}(\delta)\le\alpha}
 \ \min_{1\le i\le n}\PP_{i,\mu}\{\delta=i\}=C_{\Delta_n}(\mu,\alpha).
\end{equation}
When \(m\ge n-1\), the maximizing Gram matrix in either optimization is
\(\Delta_n\) and no other.
\end{theorem}

\begin{proof}
First we record the properties of \(\ell_\mu\). For every \(x\in\R\),
\[
 0<\ell_\mu(x)\le e^{(x_+)^2/2},\qquad
 \ell_\mu'(x)=\int\lambda e^{\lambda x-\lambda^2/2}\,d\mu(\lambda)>0,
\]
where the bound follows by maximizing \(\lambda x-\lambda^2/2\) over
\(\lambda\ge0\), and differentiation under the integral is justified because
\(\lambda e^{\lambda x-\lambda^2/2}\) is bounded uniformly in \(\lambda\ge0\)
and in \(x\) in any bounded interval. Tonelli's theorem and
\(\ell_\mu(M_G)\le\sum_i\ell_\mu(X^G_i)\) give
\begin{equation}\label{eq:id-integrability}
 \E\ell_\mu(N)=1\quad(N\sim\cN(0,1)),\qquad
 \E B_G\le\frac1n\sum_i\E\ell_\mu(X^G_i)=1.
\end{equation}

Now let \(r_i(y)\) be the conditional probability that a rule outputs label
\(i\), and \(p(y)=\sum_{i=1}^nr_i(y)\) its probability of outputting a label.
Its average success is
\[
 \frac1n\sum_i\PP_{i,\mu}\{\delta=i\}
 =\frac1n\E_0\sum_i\ell_\mu(\ip{u_i}{Y})\,r_i(Y)\le\E_0[B_Gp(Y)],
 \qquad 0\le p\le1,\quad\E_0p\le\alpha,
\]
and allocating each accepted observation to a label maximizing \(\ip{u_i}{Y}\) attains the
inequality. For any such \(p\), Tonelli's theorem gives
\begin{equation}\label{eq:id-envelope-bound}
 \E[B_Gp]=\int_0^\infty\E\bigl[p\,\ind_{\{B_G>s\}}\bigr]\,ds
 \le\int_0^\infty\min\{\alpha,\PP(B_G>s)\}\,ds.
\end{equation}
The bound is attained by accepting the upper tail of \(B_G\): choose \(c\) with
\(\PP(B_G>c)\le\alpha\le\PP(B_G\ge c)\), and let \(p=1\) on \(\{B_G>c\}\),
\(p=0\) on \(\{B_G<c\}\), and \(p\) equal on \(\{B_G=c\}\) to the constant in \([0,1]\) that makes
\(\E_0p=\alpha\); this gives equality in the integrand for almost every
\(s\). At \(\alpha=1\), take \(p=1\).

In fact no randomization is needed. The maximum \(M_G\) has no atoms, since
\(\{M_G=t\}\subset\bigcup_i\{X^G_i=t\}\), so \eqref{eq:id-null-quantile} has a
solution. Since \(\ell_\mu\) is strictly increasing,
\(\{M_G>t_G(\alpha)\}=\{B_G>\ell_\mu(t_G(\alpha))/n\}\), so the stated rule
accepts the upper \(\alpha\)-tail of \(B_G\). Ties between identical
directions can be assigned to any fixed maximizing label. This proves the
formula in \eqref{eq:id-frontier}.

\Cref{thm:main} gives \(B_G\le_{\mathrm{st}}B_{\Delta_n}\), and so the comparison in \eqref{eq:id-frontier}. For strictness, suppose \(G\ne\Delta_n\)
and choose \(x_0>0\) with \(\PP\{M_{\Delta_n}>x_0\}<\alpha\). For
\(s\in[\ell_\mu(x_0)/n,\ell_\mu(x_0+1)/n]\), both tails
\(\PP(B_G>s)\le\PP(B_{\Delta_n}>s)\) are below \(\alpha\), and the change of
variable \(s=\ell_\mu(x)/n\) gives
\begin{equation}\label{eq:id-strictness}
 C_{\Delta_n}(\mu,\alpha)-C_G(\mu,\alpha)
 \ge\frac1n\int_{x_0}^{x_0+1}\bigl[F_G(x)-D_n(x)\bigr]\ell'_\mu(x)\,dx>0,
\end{equation}
by \eqref{eq:main-equality}. Only thresholds above \(x_0>0\) enter, so the equality case of \cref{thm:main} at positive thresholds suffices.

The stated rule does not depend on \(\mu\), so it is optimal for every point
mass at \(\lambda>0\). It is also optimal at \(\lambda=0\): there the active and inactive
laws coincide, so every rule has average success
\(\frac1n\PP_0\{\delta\ne0\}\le\alpha/n\), which the stated rule attains.
Integrating its success probabilities over \(\mu\), by Tonelli's theorem,
proves \eqref{eq:id-linearity}.

Finally let \(m\ge n-1\) and realize \(\Delta_n\) by a regular simplex. Its
permutations are induced by orthogonal maps, so the rule of the theorem succeeds with
the same probability under every label, and ties have probability zero. The
smallest success probability of any rule is at most its average, so
\eqref{eq:id-frontier} proves \eqref{eq:id-label-minimax} and its attainment,
and \eqref{eq:id-strictness} proves uniqueness.
\end{proof}

For a deterministic amplitude \(\lambda>0\) and \(0<\alpha<1\), we have
\(\ell_\mu(x)=e^{\lambda x-\lambda^2/2}\) in \eqref{eq:id-profile}, and the
optimal rule gives a truncated exponential moment of the maximum,
\[
 C_G(\lambda,\alpha)=\frac{e^{-\lambda^2/2}}n\,
 \E\bigl[e^{\lambda M_G}\ind_{\{M_G>t_G(\alpha)\}}\bigr].
\]
For the simplex it is explicit in terms of \(D_n\). Write
\(t_n(\alpha)=t_{\Delta_n}(\alpha)\), the unique positive solution of
\(D_n(t_n(\alpha))=1-\alpha\) (\cref{lem:compact,lem:recurrence}).
Integration by parts gives
\begin{equation}\label{eq:id-simplex-value}
 C_{\Delta_n}(\lambda,\alpha)
 =\frac{e^{-\lambda^2/2}}n\Bigl[\alpha e^{\lambda t_n(\alpha)}
  +\lambda\int_{t_n(\alpha)}^\infty e^{\lambda t}\bigl[1-D_n(t)\bigr]\,dt\Bigr].
\end{equation}

At \(\alpha=1\) the value is \(e^{-\lambda^2/2}\E e^{\lambda M_{\Delta_n}}/n\), the decoding
probability \eqref{eq:decoding}, so the all-accept endpoint of
\cref{thm:identification} is the Weak Simplex Conjecture, already known with
its equality case (\cref{sec:integrated}). At \(\alpha=0\) every rule has
zero success, since each active law is absolutely continuous with respect to
\(\PP_0\), and at amplitude \(0\) the value is \(C_G(0,\alpha)=\alpha/n\) for
every \(G\), by the proof of \cref{thm:identification}. Two refinements,
adversarial amplitudes and Bayes accuracy, are in \cref{app:identification}.

The theorem assumes that the active labels are equally weighted, and without
this assumption it is false. Take \(n=3\), \(m\ge2\), a fixed amplitude
\(\lambda>0\), \(\alpha=1\), and prior weights \((1-\varepsilon)/2\),
\((1-\varepsilon)/2\), \(\varepsilon\). For the regular triangle, every rule
satisfies
\[
 \PP_1\{\delta=1\}+\PP_2\{\delta=2\}
 \le1+\operatorname{TV}(\PP_1,\PP_2)=2\Phi\Bigl(\frac{\lambda|u_1-u_2|}2\Bigr)
 =2\Phi\Bigl(\frac{\sqrt3\lambda}2\Bigr)
\]
and \(\PP_3\{\delta=3\}\le1\), so its weighted success is at most
\((1-\varepsilon)\Phi(\sqrt3\lambda/2)+\varepsilon\). Two antipodal signals for
labels \(1\) and \(2\), with the rule that outputs the sign of \(\ip{u_1}{Y}\) and never
outputs label \(3\), give \((1-\varepsilon)\Phi(\lambda)\). The latter is larger when
\(0<\varepsilon/(1-\varepsilon)<\Phi(\lambda)-\Phi(\sqrt3\lambda/2)\), which
holds for all small \(\varepsilon>0\).

\subsection{Independent sets in arc-transitive graphs}\label{sec:graphs}

Harangi and Vir\'ag \cite[Definitions~2.1 and~2.3, Remark~2.2]{HV2015} turn
Gaussian eigenvectors of a graph into independent sets: the strict local
maxima of a random eigenvector form an independent set. For \(d\ge3\) and
\(-d\le\lambda\le d\), put
\[
 \varrho_\lambda=\frac{d-2-\lambda}{2(d-1)},\qquad
 q_d(\lambda)=\PP\{X^{\Omega_\lambda}_i>0\text{ for every }i\},\qquad
 \Omega_\lambda=(1-\varrho_\lambda)I+\varrho_\lambda J\in\R^{d\times d}.
\]
As \(\lambda\) goes from \(-d\) to \(d\), \(\varrho_\lambda\) decreases from
\(1\) to \(-1/(d-1)\), so \(\Omega_\lambda\in\cE_d\). This \(q_d\) is the function of
\cite[Definition~2.9]{HV2015}. A graph is \emph{arc-transitive} if its
automorphisms act transitively on ordered pairs of adjacent vertices. Harangi
and Vir\'ag proved the following bound for cubic vertex-transitive graphs
\cite[Theorem~3]{HV2015} and for degree-four arc-transitive graphs
\cite[Theorem~2]{HV2015}, and showed that for arc-transitive graphs of every
degree it would follow from a conjectured extremal property of spherical caps
\cite[Section~2.2, Conjecture~2.13]{HV2015}, which is \cref{prob:caps} below. The following corollary proves the bound in every degree without that conjecture.

\begin{corollary}\label{cor:graph}
Let \(\mathcal G=(V,E)\) be a nonempty finite simple \(d\)-regular
arc-transitive graph with \(d\ge3\), independence number
\(\alpha(\mathcal G)\), and least adjacency eigenvalue \(\lambda_{\min}\).
Then
\[
 \frac{\alpha(\mathcal G)}{|V|}\ge q_d(\lambda_{\min}).
\]
\end{corollary}

The proof is in \cref{app:graphs}. For an eigenvalue \(\lambda\) of
\(\mathcal G\), let \(p_{\mathcal G}(\lambda)\) be the probability that a
given vertex is a strict local maximum of the normalized Gaussian eigenvector
of \(\lambda\) defined in \cref{app:graphs}. The strict local maxima of this
eigenvector form an independent set, so
\(\alpha(\mathcal G)\ge|V|\,p_{\mathcal G}(\lambda)\). For \(-d<\lambda<d\), put
\(\tau=\sqrt{(d-\lambda)/(d+\lambda)}\) and let \(B\sim\cN(0,1)\). Then
\cref{lem:localmax} gives \(\Gamma\in\cE_d\) with
\[
 p_{\mathcal G}(\lambda)=\E F_\Gamma(\tau B),\qquad
 q_d(\lambda)=\E D_d(\tau B).
\]
At \(\lambda=\lambda_{\min}>-d\) we have \(\tau^2\ge1/(d-1)\), so the smoothed
comparison \eqref{eq:predecessor} already gives
\(p_{\mathcal G}(\lambda_{\min})\ge q_d(\lambda_{\min})\). At
\(\lambda_{\min}=-d\) both sides equal \(1/2\). \Cref{thm:main} gives
\(p_{\mathcal G}(\lambda)\ge q_d(\lambda)\) at every eigenvalue
(\cref{prop:graph-local}), but \(q_d\) is nonincreasing (by Slepian's
inequality, since \(\varrho_\lambda\) decreases in \(\lambda\)), so this does
not improve \cref{cor:graph}.

\section{Open problems}\label{sec:open}

We close with two problems. The first is a spherical strengthening of \cref{thm:main}, which remains open. The second asks for the optimal directions when they must lie in a space too small for a regular simplex.

Let \(\nu\) denote the uniform probability measure on the unit sphere \(S^{n-2}\subset\R^{n-1}\), and let \(W\) be a standard Gaussian vector in \(\R^{n-1}\); its direction \(W/|W|\) has law \(\nu\) and is independent of its length \(|W|\). If \(u_1,\ldots,u_n\in S^{n-2}\) have Gram matrix \(G\), then \(X_i=\ip{u_i}{W}\) defines \(X\sim\cN(0,G)\), and conditioning on \(|W|\) gives, for \(t>0\),
\begin{equation}\label{eq:caps}
 1-F_G(t)=\E\,\nu\Bigl\{\theta:\max_i\ip{u_i}{\theta}>\frac t{|W|}\Bigr\}.
\end{equation}
The set on the right is a union of \(n\) equal spherical caps centered at the \(u_i\). Thus \cref{thm:main} says that, on average over the cap radius, the union has the largest measure when the centers form a regular simplex. Harangi and Vir\'ag \cite[Conjecture~2.13]{HV2015}, and later Balitskiy, Karasev, and Tsigler \cite[Conjecture~2.2]{BKT2017}, conjectured that the regular simplex maximizes the measure of the union for every cap radius (\cref{prob:caps}). By \eqref{eq:caps}, an affirmative answer gives \eqref{eq:main} whenever \(\rank G\le n-1\), and hence for every \(G\) by \cref{lem:kernel-reduction} and \eqref{eq:slepian}.

\begin{problem}[{\cite[Conjecture~2.13]{HV2015}, \cite[Conjecture~2.2]{BKT2017}}]\label{prob:caps}
Let \(n\ge5\), and let \(v_1,\ldots,v_n\) be the vertices of a regular simplex inscribed in \(S^{n-2}\). Is it true that for every \(c\in(0,1)\) and all unit vectors \(u_1,\ldots,u_n\in\R^{n-1}\),
\[
 \nu\Bigl\{\theta:\max_i\ip{u_i}{\theta}>c\Bigr\} \le\nu\Bigl\{\theta:\max_i\ip{v_i}{\theta}>c\Bigr\}?
\]
\end{problem}

Only the range \(1/(n-1)<c<\sqrt{(n-2)/(2(n-1))}\) is in question. For \(c\le1/(n-1)\) the caps of the simplex cover the sphere up to finitely many points, since \(\{y:\max_i\ip{v_i}{y}\le1/(n-1)\}\) is the simplex with vertices \(-v_1,\ldots,-v_n\), which meets the sphere only at its vertices. For \(c\ge\sqrt{(n-2)/(2(n-1))}\) they are pairwise disjoint, since a point \(\theta\) in two of them would give \(2c<\ip{v_i+v_j}{\theta}\le|v_i+v_j|=\sqrt{2(n-2)/(n-1)}\). In the first case the simplex caps have full measure, and in the second their union has measure \(n\) times that of one cap, the most any \(n\) caps can have.

For \(n\le3\) the question is elementary (for \(n=3\) the caps are arcs of a circle, the case \(d=2\) discussed after \cref{prob:rank}). For \(n=4\) the answer is yes: a standard consequence of the moment theorem of Fejes T\'oth \cite[p.~27, (4)]{FejesToth1950} is that, for every radius, the union of \(f\) equal caps on the two-sphere has the largest area when \(f\in\{4,6,12\}\) and the centers are the vertices of the tetrahedron, the octahedron, or the icosahedron (see \cite[Section~2.2]{HV2015} and \cite[p.~2]{BKT2017}). As Balitskiy, Karasev, and Tsigler point out \cite[p.~3]{BKT2017}, the moment argument of Fejes T\'oth works only when no point lies in three caps of the regular configuration, while for \(n\ge5\) three caps of the regular simplex overlap when \(c\) is close to \(1/(n-1)\).

\Cref{thm:main} places no restriction on the rank of \(G\), and the extremal
matrix \(\Delta_n\) has rank \(n-1\): the regular simplex needs \(n-1\)
dimensions for \(n\) directions. When the directions must lie in \(\R^d\)
with \(d<n-1\), we do not know the optimal geometry. In our preceding paper we
posed this question for exponential moments \cite[Open problem~8.2]{Mulgund2026},
where by \eqref{eq:decoding} it asks for the best equal-energy code of \(n\)
signals in dimension \(d\). \Cref{prob:rank} asks it for the distribution
function.

\begin{problem}\label{prob:rank}
Fix integers \(d\ge3\) and \(n\ge d+2\), and fix \(t>0\). Determine
\[
 \min\{F_G(t):G\in\cE_n,\ \rank G\le d\},
\]
and characterize the minimizing matrices.\footnote{The author has partial
results on this problem and would be glad to hear from anyone interested in
collaborating on it.}
\end{problem}

Geometrically, \cref{prob:rank} asks which unit vectors
\(u_1,\ldots,u_n\in\R^d\) minimize the Gaussian measure of
\(\{x:\ip{u_i}{x}\le t\text{ for every }i\}\). The minimum is attained by
compactness, and by \cref{thm:main} it is strictly larger than \(D_n(t)\),
because \(\rank\Delta_n=n-1>d\). For \(n\le d+1\) the minimum is \(D_n(t)\), attained only at \(\Delta_n\), because \(\rank\Delta_n=n-1\le d\). For \(d\le2\) the problem is
elementary. When \(d=1\) the directions are \(\pm u\), and using both signs
gives \(2\Phi(t)-1\). When \(d=2\), conditioning on the radius of a standard
Gaussian vector reduces the question to covering as much of a circle as
possible by \(n\) arcs of equal length, and equally spaced centers do this for
every radius, so the regular \(n\)-gon is optimal. When \(d=3\) and
\(n\in\{6,12\}\), integrating the cap bound of Fejes T\'oth recalled after \cref{prob:caps} over the radius shows that unit vectors at the vertices of the octahedron and of the icosahedron are optimal. We do not know the optimal configurations in any other case.

\section*{Acknowledgments}

The author thanks Natasha Devroye, Gy\"orgy Tur\'an, and Milo\v{s} \v{Z}efran
for helpful conversations and feedback.

\appendix

\section{Equivalence with the circumscribed-simplex conjecture}
\label[appendix]{app:equivalence}

We first complete the proof of \cref{cor:geometric}, whose inequality was proved after its statement.

\begin{proof}[Proof of the equality case of \cref{cor:geometric}]
Write \(K=\{y:\ip{v_i}{y}\le b_i,\ 1\le i\le n\}\) with unit outer facet normals \(v_i\) and offsets \(b_i\ge t\), let \(G\) be the Gram matrix of the normals, and suppose \(\gamma_{n-1}(K)=\gamma_{n-1}(T_t)\). Then both inequalities in \eqref{eq:geometric-chain} are equalities: \(F_G(t)=D_n(t)\), and \(K\setminus\{y:\ip{v_j}{y}\le t\text{ for every }j\}\) is \(\gamma_{n-1}\)-null. \Cref{thm:main} gives \(G=\Delta_n\), so the normals form a regular simplex. If \(b_i>t\) for some \(i\), the point \((t+\varepsilon)v_i\) with small \(\varepsilon>0\) satisfies every inequality defining \(K\) strictly but violates \(\ip{v_i}{y}\le t\). A small ball about this point lies in that null set, a contradiction. Thus every \(b_i=t\). Finally, two spanning families of vectors with the same Gram matrix differ by an orthogonal map, since the correspondence between them extends linearly and preserves inner products. The converse is clear.
\end{proof}

The rest of this appendix proves \cref{prop:equivalence}: the non-strict
circumscribed-simplex statement implies \eqref{eq:main}. The reduction does
not use the proof of \cref{thm:main}. It follows the argument sketched in
\cite[p.~4]{BKT2017}: \cref{lem:kernel-reduction} lowers the correlations
until \(0\in\conv\{v_1,\ldots,v_n\}\), and \cref{lem:positive-kernel-approx}
approximates the resulting polyhedra by simplices.

\begin{proposition}\label{prop:equivalence}
Fix \(n\ge2\). The following are equivalent.
\begin{enumerate}[label=(\roman*),leftmargin=*]
 \item For every \(t>0\), the regular simplex circumscribed about
 \(B(0,t)\subset\R^{n-1}\) minimizes standard Gaussian measure among all
 simplices containing \(B(0,t)\).
 \item \(F_G(t)\ge D_n(t)\) for every \(G\in\cE_n\) and every \(t\in\R\).
\end{enumerate}
\end{proposition}

We use Slepian's inequality in the following common-threshold form:
\begin{equation}\label{eq:slepian}
 \Gamma,G\in\cE_n,\quad \Gamma_{ij}\le g_{ij}\ (i\ne j)
 \quad\Longrightarrow\quad F_\Gamma(t)\le F_G(t)\quad(t\in\R).
\end{equation}
Slepian \cite[Section~1.3, p.~468, Lemma~1]{Slepian1962} states it at \(t=0\).
At every threshold it follows from Plackett's covariance derivative
\cite{Plackett1954}. For positive definite \(\Gamma,G\) the segment
\(C_s=(1-s)\Gamma+sG\) is positive definite, and
\cite[pp.~352--353, (4)--(6)]{Plackett1954} gives, with \(X\sim\cN(0,C_s)\),
\[
 \frac{d}{ds}F_{C_s}(t)=\sum_{i<j}(g_{ij}-\Gamma_{ij})\,
 \varphi_2\bigl(t,t;(C_s)_{ij}\bigr)\,
 \PP\{X_\ell\le t\ (\ell\ne i,j)\mid X_i=X_j=t\}\ge0.
\]
Integrating over \(s\in[0,1]\) proves \eqref{eq:slepian}. For positive
semidefinite \(\Gamma,G\), apply this to \((1-\varepsilon)\Gamma+\varepsilon I\) and
\((1-\varepsilon)G+\varepsilon I\), whose off-diagonal entries keep their
order, and let \(\varepsilon\downarrow0\) using \cref{lem:compact}.

\begin{lemma}\label{lem:kernel-reduction}
If \(G\in\cE_n\) and \(G\ne J\), there are \(\Gamma\in\cE_n\) and
\(a\in[0,\infty)^n\) with \(\sum_ia_i=1\) such that
\[
 \Gamma_{ij}\le g_{ij}\quad\text{for every }i,j,\qquad \Gamma a=0.
\]
\end{lemma}

\begin{proof}
Realize \(G\) by unit vectors \(v_i\), and let \(p\) be the point of their
convex hull closest to the origin. If \(p=0\), choose convex coefficients
representing it and take \(\Gamma=G\). Otherwise set \(c=|p|>0\), \(u=p/c\), and
\(\zeta_j=\ip{v_j}{u}\). The nearest-point condition \(\ip{p}{v_j-p}\ge0\)
gives \(\zeta_j\ge c\). If \(c=1\), every \(v_j\) equals \(u\), contradicting
\(G\ne J\); so \(c<1\).

Let \(\mathcal I=\{i:\zeta_i=c\}\). Any convex representation
\(p=\sum_ia_iv_i\) uses only indices in \(\mathcal I\), since
\(c=\sum_ia_i\zeta_i\) and every \(\zeta_i\ge c\). For \(i\in\mathcal I\) put
\[
 w_i=\frac{v_i-cu}{\sqrt{1-c^2}},\qquad
 v_i(s)=\sqrt{1-s^2}\,w_i+su,\qquad 0\le s\le c,
\]
and keep \(v_j(s)=v_j\) for \(j\notin\mathcal I\). The \(w_i\) are unit
vectors orthogonal to \(u\), \(v_i(c)=v_i\), and
\(\sum_{i\in\mathcal I}a_iw_i=0\). For \(i,k\in\mathcal I\),
\[
 \frac{d}{ds}\ip{v_i(s)}{v_k(s)}=2s\bigl(1-\ip{w_i}{w_k}\bigr)\ge0.
\]
For \(i\in\mathcal I\) and \(j\notin\mathcal I\), let
\(d_{ij}=\ip{w_i}{v_j}\). Since \(w_i\perp u\),
\(d_{ij}=\ip{w_i}{v_j-\zeta_ju}\le|v_j-\zeta_ju|=\sqrt{1-\zeta_j^2}\), and
\[
 \frac{d}{ds}\ip{v_i(s)}{v_j}
 =\zeta_j-\frac{s\,d_{ij}}{\sqrt{1-s^2}}
 \ge\zeta_j-\frac{s\sqrt{1-\zeta_j^2}}{\sqrt{1-s^2}}\ge0,
\]
where the last inequality follows by squaring and using \(s\le c\le\zeta_j\).
Thus moving \(s\) from \(c\) down to \(0\) decreases every affected inner
product. The Gram matrix \(\Gamma\) of the vectors \(v_i(0)\) satisfies \(\Gamma\le G\)
entrywise, and \(\Gamma a=0\) because
\(\sum_{i\in\mathcal I}a_iv_i(0)=\sum_{i\in\mathcal I}a_iw_i=0\).
\end{proof}

The second lemma approximates such a \(\Gamma\) by matrices whose normals form
simplices.

\begin{lemma}\label{lem:positive-kernel-approx}
If \(\Gamma\in\cE_n\), \(\Gamma a=0\), \(a\ge0\), and \(\sum_ia_i=1\), then \(\Gamma\) is a
limit of rank-\((n-1)\) correlation matrices whose kernels are spanned by
vectors with strictly positive entries.
\end{lemma}

\begin{proof}
For \(0<\varepsilon<1\) let
\[
 a_\varepsilon=(1-\varepsilon)a+\frac\varepsilon n\one,\qquad
 P_\varepsilon=I-\frac{a_\varepsilon a_\varepsilon^\top}{|a_\varepsilon|^2},
 \qquad
 \Sigma_\varepsilon=P_\varepsilon\Gamma P_\varepsilon+\varepsilon P_\varepsilon.
\]
The identity
\(x^\top \Sigma_\varepsilon x=(P_\varepsilon x)^\top \Gamma(P_\varepsilon x)
+\varepsilon|P_\varepsilon x|^2\) shows that \(\Sigma_\varepsilon\succeq0\) and
\(\ker \Sigma_\varepsilon=\operatorname{span}\{a_\varepsilon\}\). Each diagonal
entry is positive, \((\Sigma_\varepsilon)_{ii}\ge\varepsilon(P_\varepsilon)_{ii}>0\),
because \(a_\varepsilon\) has at least two positive entries. Set
\[
 T_\varepsilon=\diag\bigl((\Sigma_\varepsilon)_{ii}^{-1/2}\bigr),\qquad
 \Gamma_\varepsilon=T_\varepsilon \Sigma_\varepsilon T_\varepsilon.
\]
Then \(\Gamma_\varepsilon\in\cE_n\) has rank \(n-1\) and kernel spanned by
\(T_\varepsilon^{-1}a_\varepsilon\), which has positive entries. Since
\(\Gamma a=0\), the projection \(P_0=I-aa^\top/|a|^2\) satisfies \(P_0\Gamma P_0=\Gamma\), so
\(\Sigma_\varepsilon\to \Gamma\), \(T_\varepsilon\to I\), and \(\Gamma_\varepsilon\to \Gamma\).
\end{proof}

\begin{proof}[Proof of \cref{prop:equivalence}]
The implication (ii)\(\Rightarrow\)(i) is the non-strict part of the proof of
\cref{cor:geometric}. For the converse, assume (i) and take \(G\in\cE_n\). If
\(G=J\), then \(\Delta_n\le J\) entrywise and \eqref{eq:slepian} applies.
Otherwise \cref{lem:kernel-reduction,lem:positive-kernel-approx} give
\(\Gamma\le G\) entrywise and \(\Gamma_\varepsilon\to \Gamma\), each \(\Gamma_\varepsilon\) of rank
\(n-1\) with a kernel vector of positive entries.

Realize \(\Gamma_\varepsilon\) by unit vectors \(v_1,\ldots,v_n\) spanning
\(\R^{n-1}\). Their only linear dependence has positive coefficients, so they
are affinely independent (a dependence with coefficients summing to zero must
vanish), and their convex hull is a simplex with the origin in its interior.
For \(t>0\), the set \(\{y:\ip{v_i}{y}\le t,\ 1\le i\le n\}\) is \(t\) times
the polar of that simplex, so it is a bounded simplex. Since \(|v_i|=1\),
every \(y\) with \(|y|\le t\) satisfies \(\ip{v_i}{y}\le t\), so this simplex
contains \(B(0,t)\), and (i) gives
\[
 F_{\Gamma_\varepsilon}(t)=\gamma_{n-1}\{y:\ip{v_i}{y}\le t\text{ for every }i\}
 \ge D_n(t).
\]
Letting \(\varepsilon\downarrow0\) by \cref{lem:compact} and then applying
\eqref{eq:slepian} gives \(F_G(t)\ge F_\Gamma(t)\ge D_n(t)\). For \(t\le0\) the
right side vanishes.
\end{proof}

Strict inequalities do not pass through the approximation in
\cref{lem:positive-kernel-approx}. This is why the equality statement
\eqref{eq:main-equality} is proved separately, in \cref{sec:mass}.

\section{Further identification criteria}\label[appendix]{app:identification}

We keep the notation of \cref{sec:identification}.

\subsection{Adversarial amplitudes}

When the amplitude law is not known, we ask for rules that succeed against
every law in a set \(\mathcal A\) and every label.

\begin{corollary}\label{cor:robust-amplitude}
Let \(n\ge2\), \(m\ge n-1\), and \(0<\alpha\le1\). Let \(\mathcal A\) be a set
of probability laws on \([0,\infty)\) containing a law \(\mu_0\) with
\(\mu_0\le_{\mathrm{st}}\mu\) for every \(\mu\in\mathcal A\). If
\(\mu_0((0,\infty))>0\), then
\begin{equation}\label{eq:id-robust}
 \sup_{u_1,\ldots,u_n\in S^{m-1}}\
 \sup_{\delta:\operatorname{FA}(\delta)\le\alpha}\
 \inf_{\mu\in\mathcal A}\ \min_{1\le i\le n}\PP_{i,\mu}\{\delta=i\}
 =C_{\Delta_n}(\mu_0,\alpha).
\end{equation}
The rule of \cref{thm:identification} for the simplex attains the value, and
\(\Delta_n\) is the only maximizing Gram matrix. In particular, if an
adversary chooses a label and an amplitude \(\lambda\ge\lambda_0>0\), the value is
\(C_{\Delta_n}(\lambda_0,\alpha)\).
\end{corollary}

\begin{proof}
For the upper bound, fix unit vectors \(u_1,\ldots,u_n\) with Gram matrix
\(G\), and take \(\mu=\mu_0\). For every rule \(\delta\) with
\(\operatorname{FA}(\delta)\le\alpha\),
\[
 \min_i\PP_{i,\mu_0}\{\delta=i\}\le\frac1n\sum_i\PP_{i,\mu_0}\{\delta=i\}
 \le C_G(\mu_0,\alpha)\le C_{\Delta_n}(\mu_0,\alpha),
\]
and the last inequality is strict for \(G\ne\Delta_n\) by
\cref{thm:identification}.

For attainment, use the simplex rule of \cref{thm:identification}.
Conditional on the amplitude \(\lambda\), its correct-decision event under label
\(i\) is, up to null sets,
\begin{equation}\label{eq:id-nested-events}
 \bigl\{\ip{u_i}{Z}>t_n(\alpha)-\lambda\bigr\}\cap\bigcap_{j\ne i}
 \bigl\{\ip{u_i-u_j}{Z}>-\lambda(1-g_{ij})\bigr\};
\end{equation}
at \(\alpha=1\) omit the first event. These events increase with \(\lambda\) because
\(1-g_{ij}>0\) for simplex directions, and by symmetry their probabilities are
the same for every label. Since \(\mu_0\le_{\mathrm{st}}\mu\), every label
succeeds with probability at least \(C_{\Delta_n}(\mu_0,\alpha)\) for every
\(\mu\in\mathcal A\).
\end{proof}

The hypothesis \(\mu_0((0,\infty))>0\) cannot be dropped. If \(\mathcal A\) contains deterministic amplitudes approaching zero, every rule has smallest success at most \(\alpha/n\).
Indeed, for fixed \(i\),
\[
 \bigl|\PP_{i,\lambda}\{\delta=i\}-\PP_0\{\delta=i\}\bigr|
 \le\E_0\bigl|e^{\lambda\ip{u_i}{Z}-\lambda^2/2}-1\bigr|\le\sqrt{e^{\lambda^2}-1}
 \longrightarrow0,
\]
and \(\min_i\PP_0\{\delta=i\}\le\alpha/n\), since
\(\sum_i\PP_0\{\delta=i\}\le\alpha\). Outputting a uniformly random label with probability
\(\alpha\), and \(0\) otherwise, attains this value for every geometry.

\subsection{Bayes accuracy}

Suppose inactivity has prior probability \(\pi_0\in[0,1)\) and each active
label has prior \(\pi=(1-\pi_0)/n\). A rule with output probabilities
\(r_0,\ldots,r_n\) is correct with probability
\[
 \pi_0\E_0r_0(Y)+\pi\sum_i\E_0\bigl[\ell_\mu(\ip{u_i}{Y})\,r_i(Y)\bigr]
 \le\E_0\max\Bigl\{\pi_0,\ \pi\max_i\ell_\mu(\ip{u_i}{Y})\Bigr\},
\]
with equality for the rule that outputs a maximizing term. Since \(\ell_\mu\) is
increasing, the optimal Bayes accuracy is
\begin{equation}\label{eq:id-bayes}
 \mathcal B_G(\mu,\pi_0)=\E\max\{\pi_0,\pi\ell_\mu(M_G)\}.
\end{equation}

\begin{corollary}\label{cor:bayes}
If \(\mu((0,\infty))>0\), then for every \(\pi_0\in[0,1)\),
\(\mathcal B_G(\mu,\pi_0)\le\mathcal B_{\Delta_n}(\mu,\pi_0)\), with equality if
and only if \(G=\Delta_n\).
\end{corollary}

\begin{proof}
Apply \cref{cor:integrated} to \(\psi=\max\{\pi_0,\pi\ell_\mu\}\), which is
nondecreasing with \(0\le\psi\le\pi_0+\pi\ell_\mu\); both expectations are
finite by \eqref{eq:id-integrability}. The function \(\ell_\mu\) is strictly increasing and unbounded: some \([\lambda_1,\lambda_2]\subset(0,\infty)\) has positive
\(\mu\)-mass, and
\(\ell_\mu(x)\ge\mu([\lambda_1,\lambda_2])e^{\lambda_1x-\lambda_2^2/2}\) for
\(x\ge0\). Since \(\pi>0\), there is \(x_1>0\) with
\(\pi\ell_\mu(x_1)\ge\pi_0\), and \(\psi(x)=\pi\ell_\mu(x)\) is strictly
increasing for \(x\ge x_1\). So \cref{cor:integrated} gives strictness for
\(G\ne\Delta_n\).
\end{proof}

\section{Local maxima at every eigenvalue}\label[appendix]{app:graphs}

In this appendix we prove \cref{cor:graph} from the smoothed comparison
\eqref{eq:predecessor}, and then use \cref{thm:main} to prove
\(p_{\mathcal G}(\lambda)\ge q_d(\lambda)\) at every eigenvalue \(\lambda\)
(\cref{prop:graph-local}), where \(p_{\mathcal G}\) is the local-maximum
probability defined below.

Let \(d\ge3\), and let \(\mathcal G=(V,E)\) be a nonempty finite simple
\(d\)-regular arc-transitive graph with adjacency matrix \(A\). Its
eigenvalues lie in \([-d,d]\), since
\(|x^\top Ax|\le\sum_{\{u,v\}\in E}(x_u^2+x_v^2)=d|x|^2\). For an eigenvalue
\(\lambda\), choose an orthonormal basis \(f_1,\ldots,f_k\) of its eigenspace
and let \(X=\sum_{\ell=1}^kZ_\ell f_\ell\) with independent standard Gaussian
coefficients. Every automorphism preserves the eigenspace and acts
orthogonally on it, so the law of \(X\) is invariant
\cite[Definition~2.1, Remark~2.2]{HV2015}. Arc-transitivity implies
vertex-transitivity, because every vertex has a neighbor. So the coordinate
variances are equal, with sum \(\E|X|^2=k>0\), and we rescale \(X\) to make
each equal to one. Put
\[
 p_{\mathcal G}(\lambda)=\PP\{X_v>X_u\text{ for every }u\sim v\},
\]
which does not depend on the vertex \(v\). Realize the matrix
\(\Omega_\lambda\) in the definition of
\(q_d(\lambda)\) as the Gram matrix of unit vectors
\(w_1,\ldots,w_d\in\R^d\). Since \(Z/|Z|\) is uniform on the sphere for
\(Z\sim\gamma_d\), \(q_d(\lambda)\) is the normalized spherical volume of the
cone \(\{z:\ip{w_i}{z}>0\text{ for every }i\}\); in particular
\(q_d(-d)=1/2\) (all \(w_i\) are equal) and \(q_d(d)=0\) (\(\sum_iw_i=0\)).

\begin{lemma}\label{lem:localmax}
Let \(-d<\lambda<d\), put \(\tau=\sqrt{(d-\lambda)/(d+\lambda)}\), and let
\(B\sim\cN(0,1)\). There is \(\Gamma\in\cE_d\) such that
\begin{equation}\label{eq:localmax}
 p_{\mathcal G}(\lambda)=\E F_\Gamma(\tau B),
\end{equation}
and \(\E D_d(\tau B)=q_d(\lambda)\).
\end{lemma}

\begin{proof}
At a fixed vertex write \(X_0\) for its value and \(X_1,\ldots,X_d\) for the
values at its neighbors. The eigenvector equation and arc-transitivity give
\begin{equation}\label{eq:graph-cov}
 \sum_{i=1}^dX_i=\lambda X_0,\qquad \Cov(X_0,X_i)=c:=\frac\lambda d.
\end{equation}
Let \(s=\sqrt{1-c^2}\) and \(Y_i=(X_i-cX_0)/s\). Then \(\Var(Y_i)=1\),
\(\Cov(Y_i,X_0)=0\), and \(\sum_iY_i=(\lambda-dc)X_0/s=0\). By joint
Gaussianity \(Y\) is independent of \(X_0\), and its covariance \(\Gamma\) lies in
\(\cE_d\). The fixed vertex is a strict local maximum when \(X_i<X_0\) for
every \(i\), that is, \(Y_i<(1-c)X_0/s=\tau X_0\). Each edge difference has
variance \(2(1-c)>0\), so ties have probability zero, and conditioning on
\(X_0\sim\cN(0,1)\) gives \eqref{eq:localmax}.

For the second identity take \(B\sim\cN(0,1)\) and
\(\xi\sim\cN(0,\Delta_d)\) independent, and set
\(\widetilde W_i=(1-c)B-s\xi_i\). Then
\(\E D_d(\tau B)=\PP\{\widetilde W_i>0\text{ for every }i\}\), each \(\widetilde W_i\) has
variance \(2(1-c)\), and for \(i\ne j\)
\begin{equation}\label{eq:graph-identification}
 \operatorname{Corr}(\widetilde W_i,\widetilde W_j)
 =\frac{(1-c)^2-(1-c^2)/(d-1)}{2(1-c)}
 =\frac{d-2-dc}{2(d-1)}=\frac{d-2-\lambda}{2(d-1)}=\varrho_\lambda.
\end{equation}
Thus \(\E D_d(\tau B)=q_d(\lambda)\).
\end{proof}

As Harangi and Vir\'ag observe \cite[Section~2.2]{HV2015}, arc-transitivity
makes the covariances in \eqref{eq:graph-cov} equal. This is what makes the
thresholds \(\tau X_0\) common to all neighbors, and it is the only use of
arc-transitivity beyond vertex-transitivity.

\begin{proof}[Proof of \cref{cor:graph}]
First we extract a smoothed comparison at scale \(\tau\) from
\eqref{eq:predecessor}, applied in dimension \(d\): for \(\Gamma\in\cE_d\),
\[
 M_\Gamma+\sigma B\le_{\mathrm{st}}M_{\Delta_d}+\sigma B',\qquad
 \sigma=\frac1{\sqrt{d-1}}.
\]
If \(\tau\ge\sigma\), add independent Gaussian noise of variance
\(\tau^2-\sigma^2\) to both sides; convolution preserves stochastic order, as one
sees by integrating the distribution-function inequalities against the added
noise. Evaluating at \(0\) and using the symmetry of \(B\) gives
\begin{equation}\label{eq:graph-smoothed-cdf}
 \E F_\Gamma(\tau B)\ge\E D_d(\tau B)\qquad\Bigl(\tau^2\ge\frac1{d-1}\Bigr).
\end{equation}
For \(-d<\lambda<d\),
\[
 \tau^2=\frac{d-\lambda}{d+\lambda}\ge\frac1{d-1}
 \quad\Longleftrightarrow\quad(d-1)(d-\lambda)\ge d+\lambda
 \quad\Longleftrightarrow\quad\lambda\le d-2.
\]
For any edge \(\{u,v\}\), the Rayleigh quotient of \(\varepsilon_u-\varepsilon_v\), with
\(\varepsilon_u\) the coordinate vector of \(u\), gives
\(\lambda_{\min}\le-1\le d-2\). If \(\lambda_{\min}>-d\),
\cref{lem:localmax} and \eqref{eq:graph-smoothed-cdf} give
\(p_{\mathcal G}(\lambda_{\min})\ge q_d(\lambda_{\min})\). If
\(\lambda_{\min}=-d\), then \(c=-1\) in \eqref{eq:graph-cov}, every neighbor
value equals \(-X_0\), and \(p_{\mathcal G}(\lambda_{\min})=1/2
=q_d(\lambda_{\min})\). The strict local maxima of \(X\) form an independent
set, so by invariance
\[
 \alpha(\mathcal G)\ge\E\bigl|\{v:X_v>X_u\ \forall u\sim v\}\bigr|
 =|V|\,p_{\mathcal G}(\lambda_{\min})\ge|V|\,q_d(\lambda_{\min}).
 \qedhere
\]
\end{proof}

\Cref{thm:main} removes the restriction \(\lambda\le d-2\) from the
comparison \(p_{\mathcal G}(\lambda)\ge q_d(\lambda)\); as noted in
\cref{sec:graphs}, this does not improve \cref{cor:graph}.

\begin{proposition}\label{prop:graph-local}
Let \(\mathcal G\) be a nonempty finite simple \(d\)-regular arc-transitive
graph with \(d\ge3\), and let \(p_{\mathcal G}(\lambda)\) be the
local-maximum probability of its normalized Gaussian eigenvector defined
above. For every adjacency eigenvalue \(\lambda\) of \(\mathcal G\),
\(p_{\mathcal G}(\lambda)\ge q_d(\lambda)\).
\end{proposition}

\begin{proof}
For \(-d<\lambda<d\), \cref{lem:localmax,thm:main} give
\(p_{\mathcal G}(\lambda)=\E F_\Gamma(\tau B)\ge\E D_d(\tau B)=q_d(\lambda)\). At
\(\lambda=-d\) the probability is \(1/2=q_d(-d)\), as above, and at
\(\lambda=d\) neighbors agree almost surely and both sides are zero.
\end{proof}

For \(d=3\) the proposition is an identity: the proof of \cref{lem:localmax} gives \(\sum_iY_i=0\), so its matrix \(\Gamma\in\cE_3\) satisfies \(\Gamma\one=0\), which forces \(\Gamma=\Delta_3\) (the three rows read
\(1+\Gamma_{ij}+\Gamma_{i\ell}=0\), whose only solution is \(\Gamma_{ij}=-1/2\)), so \(p_{\mathcal G}(\lambda)=q_3(\lambda)\) at every eigenvalue. For \(d=4\) it follows from the theorem of Fejes T\'oth on four caps of the two-sphere, as Harangi and Vir\'ag show \cite[Section~2.2]{HV2015}. For cherry-transitive graphs, in which every path of length two can be mapped to every other by an automorphism, it holds with equality in every degree \cite[Definition~2.7, Proposition~2.8, and Definition~2.9]{HV2015}. For the other arc-transitive graphs of degree \(d\ge5\) it is new; for \(\lambda\le d-2\) it already follows from \eqref{eq:predecessor}, as in the proof of \cref{cor:graph}.

\section{AI disclosure}\label[appendix]{app:ai}

This appendix describes how generative AI was used to obtain
\cref{thm:main}. All dates are in 2026, in US Central time, and all models were
accessed in August and September 2026 under the names their providers gave
them.

\emph{Before the research system.} The author began working on the
circumscribed-simplex conjecture on August~14, after the first version of our
preceding paper \cite{Mulgund2026}. The first attempt was a single long conversation with
GPT-5.6 Pro in ChatGPT (OpenAI), prompted as the author had prompted it while
developing \cite{Mulgund2026}: the model kept a tree of proof routes and
periodically compressed it. Over four and a half days it produced
reformulations and partial claims, some of which it later retracted, but no
proof.

\emph{The research system.} The author then designed a research system, between
August~26 and~31, with a different aim: to build useful theory across a family
of related problems rather than to solve one problem in isolation, and to
deepen the author's own understanding of Gaussian extremal problems. Its
guiding mission is Goal~6.5 of Tao \cite[Section~6]{Tao2026AI}: ``Solve
unsolved problems, verify them to be correct, ensure they are clearly
communicated, and have them digested, accepted, and incorporated into the
definitive theory of the field.'' The system keeps a web of more than two
hundred problems in Gaussian extremal theory, linked by logical relations
(implication, equivalence, relaxation, obstruction) and structural ones
(analogy, shared method). It also proposes new problems, obtained from existing
ones by deformation, duality, or relaxation, and marks each as unverified for
novelty until it has been checked against the literature. It works on a few
central problems that appear tractable, two or three at a time, so that a
method developed for one can be tried on the others. After a problem is
solved, it continues to simplify, unify, and reinterpret the
solution.

A coordinating agent, running in OpenAI Codex on GPT-5.6 Sol and, from
September~4, on GPT-6 Astra, kept the entire research state in files. It posed
self-contained problems to ChatGPT Pro (GPT-5.6 Pro and, from September~4,
GPT-6 Pro), checked the returned arguments by having independent subagents
reconstruct them, and did mathematical work of its own. The work alternated
between two phases. In a fan-out phase, several attacks ran in parallel. After
a significant result, a compression phase reviewed everything obtained so far
and simplified, unified, and reinterpreted it into a smaller foundation for the
next round. Failed routes and counterexamples were kept.

\emph{How the result was obtained.} The circumscribed-simplex conjecture was
among the problems the system started with, on August~31. On September~1 the
statement for all correlation matrices, \eqref{eq:main}, became the primary
target at the author's insistence, as described below. GPT-5.6 Pro proved
\eqref{eq:main} for \(n=3\) and \(n=4\) on September~2 and~3, each proof
checked by an independent GPT-5.6 Pro critic. Many other attacks failed and
were recorded. For example, the Gaussian smoothing of \cite{Mulgund2026} cannot
simply be undone, since convolution with a Gaussian does not reflect
stochastic order (compare \cref{sec:smoothed}).

On the evening of September~4, shortly after the update, a GPT-6 Pro
conversation on the problem introduced the matrix \(S\) of \eqref{eq:S}, as the
second variation of \(F_G(t)\) when the vectors realizing \(G\) are lifted into
an extra dimension. On September~6 the coordinator, at the author's request,
collected equivalent reformulations and stationarity conditions, and four
focused attacks on the problem were sent to GPT-6 Pro in parallel. One of
them, on joint stationarity for five coordinates, returned a proof for \(n=5\)
early on September~7. The author then asked for a compression phase to
extract the mechanisms of that proof that could work in every dimension; the
coordinator was already preparing such a continuation. The coordinator
compressed the proof into conditions at minimizers of every rank:
\(S\succeq0\), the kernel vector \(k=S\one\), and a row bound from the
Cauchy--Schwarz inequality. Given this compression, the next GPT-6 Pro turn in
the same conversation returned a proof for all \(n\). It used log-concavity and
a dilation identity in place of an earlier Ehrhard-type concavity estimate.
Two independent GPT-6 Pro critics and the coordinator's subagents
reconstructed the proof the same morning.

Between September~7 and~22 the post-solution work compressed the proof into
its present form. The coordinator and its subagents contributed the reciprocal
form \eqref{eq:reciprocal} of the tangent inequality, the quadratic
aggregation of \cref{lem:scalar}, and the linear tilt in the proof of
\cref{thm:main}. The coordinator also recast \(S\succeq0\) as a first-order
condition along the path \eqref{eq:noise-path}, and a GPT-6 Pro consultation on
formalization supplied the proof of \cref{prop:noise} by averaging over the
added noise, which works directly at singular covariances and replaced a
regularized covariance derivative. On September~18 a literature search located
the relevant prior work. A ChatGPT Deep Research report identified
\cite[Conjecture~3.3]{BKT2017} as the closest antecedent, though it took our
statement to be strictly stronger, and pointed to McKay and Nair. Checking the
primary sources, the coordinator and its subagents found that the non-strict
statement is equivalent to that conjecture (as the coordinator had shown on
September~3), that the recurrence \eqref{eq:recurrence} is classical, and that
the case \(n=4\) follows from a theorem of Fejes T\'oth. Two GPT-6 Pro
opinions, given copies of the sources, then reviewed the attribution.

\emph{The author's contributions.} The author chose the problem, designed the
research system, and steered it throughout, with interventions at least daily
in the week before the proof. The compression phases were the author's
method, brought over from the work on \cite{Mulgund2026}: on September~3 the
author noticed that none had followed the first accepted results, and on
September~4 specified what a compression pass must do. Two ideas that organize
the proof grew out of the author's exchanges with the coordinator. On
September~1 the coordinator ranked the statement for all correlation matrices
below special cases, because it looked logically stronger than the geometric
conjecture (the two turned out to be equivalent, \cref{prop:equivalence}). The
author argued that a logically stronger statement need not be harder to prove,
and that \(\cE_n\) is a compact convex set on which methods of
\cite{Mulgund2026} might carry over. The coordinator agreed, and its scout of
the full statement reduced it to the singular matrices in \(\cE_n\); the final
proof does work over all of \(\cE_n\), through minimizers. On September~6 the
author asked whether a stochastically maximal correlation matrix must exist,
whether existence might be easier to prove than identification, and, since
proving that an improving flow cannot get stuck seemed hard, whether the
existing results could be assembled into an existence proof. The coordinator
answered by proposing to study a single worst counterexample, minimizing
\(\log(F_G(t)/D_5(t))\) jointly over the covariance and the threshold. This
joint minimum was the starting point of the proof for \(n=5\), and, with the
linear tilt described above, it is the joint minimum in the proof of
\cref{thm:main} in \cref{sec:global}. Later, the author's question whether the
noise-averaging argument, found for the formalization, gave a better proof
brought it into the written argument as the proof of \cref{prop:noise}.

The author then studied the proof, in expository editions produced by the
system and in the author's own conversations with ChatGPT. In one of these
conversations ChatGPT split the proof into the derivative bound at minimizers
(\cref{prop:interface}) and the global argument that uses it. The author had
the exposition rebuilt around this split and split further in the same way;
this organizes \cref{sec:global,sec:boundary,sec:variation,sec:mass}. In a
later conversation the author turned ChatGPT's template for the step of
\cref{sec:global} that rules out a negative difference into a general lemma
about minimizers over a compact set, whose hypotheses ChatGPT then corrected.
The author also corrected the paper's account of its relation to
\cite{Mulgund2026} and directed the writing.

The technical steps of the proof were produced by the models, as described
above. The overall strategy, however, developed in conversation: the author's
questions and reformulations were answered by the coordinator's, and these in
turn shaped the problems posed to GPT-6 Pro. So it is difficult to say
precisely where each idea originated.

\emph{Formalization, writing, and checking.} A Codex agent on GPT-6 Astra, with
three helper agents, formalized \cref{thm:main} in Lean on September~8,
following a plan written by GPT-6 Pro. The formal statements are exactly
\eqref{eq:main} and \eqref{eq:main-equality}; independent agents reviewed the
formalization, and Lean's kernel checker accepts it. GPT-6 Pro wrote a first
draft of this article, and Claude Opus 5.5 (Anthropic) rewrote and revised it
under the author's direction; fresh model instances reviewed it. Every
reference locator was checked against a copy of the source.

\emph{Availability.} The author plans to release the research system for
general use once it has been cleared of personal and sensitive information.
Its source will accompany a future revision of this preprint.

\bibliographystyle{alpha}
\bibliography{references}
\end{document}